\documentclass[preprint,11pt]{elsarticle}

\usepackage[margin=1in]{geometry}
\usepackage{amsmath,amssymb,amsthm}
\usepackage{enumitem}
\usepackage[colorlinks=true,linkcolor=blue,citecolor=blue,urlcolor=blue]{hyperref}

\newcommand{\dif}{\,\mathrm{d}}
\newcommand{\Pc}{\mathcal{P}}

\newtheorem{theorem}{Theorem}[section]
\newtheorem{proposition}[theorem]{Proposition}
\newtheorem{lemma}[theorem]{Lemma}
\newtheorem{corollary}[theorem]{Corollary}
\theoremstyle{definition}
\newtheorem{definition}[theorem]{Definition}
\theoremstyle{remark}
\newtheorem{remark}[theorem]{Remark}

\begin{document}

\begin{frontmatter}

\title{Generalized Mean-Reflected McKean--Vlasov BSDEs:
  Well-Posedness and a Markovian PDE Representation}

\author{Ruisen Qian}
\address{School of Mathematical Sciences, Fudan University,
  Shanghai, 200433, China.}
\ead{rsqian19@fudan.edu.cn}

\begin{abstract}
This paper investigates a class of generalized mean-reflected
  McKean--Vlasov type backward stochastic differential equations (BSDEs).
  Our new framework combines a mean reflection constraint on the solution's
  expectation with a generalized integral with respect to a continuous
  non-decreasing process. We establish the existence and uniqueness of the
  solution. The uniqueness is derived via stability estimates, while the
  existence is proved by employing a penalization method combined with a
  smooth approximation of the obstacle. We further study a Markovian
  case, where the equation reduces to a nonlocal parabolic PDE with a
  mean constraint. We establish the
  well-posedness and the probabilistic representation of this problem,
  and we describe the regularity of the multiplier in terms of the
  regularity of the obstacle.
\end{abstract}


\end{frontmatter}

\section{Introduction}
\label{sec:intro5}

Since the seminal work of Pardoux and Peng \cite{pardoux_1990} on the existence and uniqueness of adapted solutions to nonlinear backward stochastic differential equations (BSDEs), given by
\[
Y_t ={} \xi + \int_t^T f(s, Y_s, Z_s) \dif s - \int_t^T Z_s \dif B_s,
\]
the theory of BSDEs has been developed extensively and applied to stochastic control, mathematical finance, and partial differential equations. An important extension is the reflected backward stochastic differential equation (RBSDE), introduced by El Karoui et al. \cite{karoui_1997}. In an RBSDE, the first component of the solution is constrained to stay above a given obstacle process $L$. This constraint introduces an increasing process $K$ that ensures the Skorokhod condition
\[
\int_0^T (Y_t - L_t) \dif K_t ={} 0
\]
is satisfied, thereby establishing a connection with optimal stopping problems and obstacle problems for parabolic partial differential equations (PDEs). Since then, reflected BSDEs have been extended in several directions. Multi-dimensional RBSDEs, oblique reflection, and their links with optimal switching problems have been studied in \cite{tang_2010,hamadène_2010,tang_2011,hu_2015,richou_2020,ourkiya_2023,shi_2024}. Penalization methods and discrete-time approximations for reflected or obliquely reflected BSDEs have been studied in \cite{chassagneux_2012,chassagneux_2019}. Quadratic reflected BSDEs and their PDE obstacle problems have been studied in \cite{cao_2020,huang_2020,luo_2024}.

In recent years, mean-field stochastic differential equations, also known as McKean--Vlasov equations, have been studied extensively due to their applications in mean-field games and interacting particle systems. The mean-field BSDE (MF-BSDE), where the driver depends on the law of the solution $(Y, Z)$, was systematically studied by Buckdahn et al. \cite{buckdahn_2009}. Subsequent research has extended this framework to more general distributional dependence and weaker regularity assumptions \cite{li_2018}, jump-diffusion settings and associated nonlocal integral-PDEs \cite{li_2018_meanf}, random terminal times \cite{li_2026}, fractional Brownian motion \cite{douissi_2019,fu_2025}, $G$-Brownian motion \cite{sun_2020_meanf,sun_2022}, and Gaussian-process-driven distribution-dependent BSDEs \cite{fan_2025}. Interacting particle approximations and propagation of chaos for backward McKean--Vlasov systems have been studied in \cite{laurière_2022,li_2024}.

The intersection of mean-field interactions and reflection has led to the active research area of mean-field reflected BSDEs. Li \cite{li_2014} studied reflected mean-field BSDEs and their links to nonlocal obstacle problems. Cui and Zhao \cite{cui_2025} relaxed the Lipschitz assumption, while Lin and Xu \cite{lin_2025_propa} established propagation of chaos for jump-driven versions. Niu et al. \cite{niu_2025} studied  multi-dimensional mean-reflected BSDEs in possibly non-convex domains, covering both inward normal and oblique reflection.

A distinct but closely related class of problems is BSDEs with mean reflection, where the constraint is imposed on the distribution of the solution rather than its pointwise paths. This type of constraint was introduced by Briand, Elie, and Hu \cite{briand2018bsdes}, who formulated the constraint in terms of the expectation of a loss function and established well-posedness under a deterministic Skorokhod condition. This framework has since been extended in several directions: Hibon et al. \cite{hibon2017quad} studied quadratic generators, and Li \cite{li2024backward} studied double nonlinear mean reflections. To study the limiting behavior and numerical approximations, interacting particle systems and forward-backward SDEs with normal constraints in law have been analyzed, establishing propagation of chaos and connections to PDEs on the Wasserstein space \cite{briand2020forward,briand2021particles}. Penalization schemes and their convergence have also been studied for mean-field reflected BSDEs where the obstacle depends on both the solution $Y$ and its distribution $\mathbb{P}_Y$, with applications to pricing life insurance contracts with surrender options \cite{chen2022mean, qian2023multi}.

Generalized BSDEs (GBSDEs) were introduced by Pardoux and Zhang \cite{pardoux1998generalized}. These equations incorporate an additional integral term with respect to a continuous increasing process $\kappa$, given by
\[
Y_t ={} \xi + \int_t^T f(s, Y_s, Z_s) \dif s + \int_t^T g(s, Y_s) \dif \kappa_s - \int_t^T Z_s \dif B_s,
\]
which is typically interpreted as the local time of a diffusion process at the boundary. This extension provides a probabilistic representation for solutions to semi-linear PDEs with non-linear Neumann boundary conditions \cite{elhachemy2023reflected}. Subsequent research has developed this theory to accommodate discontinuous coefficients, reflection constraints, and Poisson random measure jumps under various regularity conditions (see, e.g., \cite{elhachemy2023reflected, elmansouri_2026} and references therein).

In the mean-field context, the study of generalized BSDEs remains limited. Feng \cite{feng2021generalized} studied mean-field generalized BSDEs, where the driver depends on the law of the solution $(Y, Z)$, given by
\[
Y_t ={} \xi + \int_t^T f(s, Y_s, Z_s, \nu_s) \dif s + \int_t^T g(s, Y_s) \dif \kappa_s - \int_t^T Z_s \dif B_s,
\]
where $\nu_s$ denotes the law of $(Y_s, Z_s)$. Feng established existence and uniqueness results and showed that these equations provide a probabilistic interpretation for solutions to non-local PDEs with non-linear Neumann boundary conditions through the associated McKean--Vlasov forward equation. Elhachemy \cite{elhachemy_2025} extended this framework to include reflected generalized BSDEs with jumps, considering barriers that are right-continuous with left limits (RCLL) and allowing for stochastic Lipschitz coefficients.

Generalized BSDEs provide probabilistic representations for semilinear PDEs with nonlinear Neumann boundary conditions. Pardoux and Zhang \cite{pardoux1998generalized} established this representation; Hsu \cite{hsu1985probabilistic} treated the classical Neumann problem probabilistically through reflecting Brownian motion and its boundary local time, and Carmona and Zheng \cite{carmona1994reflecting} established Neumann heat kernel estimates. Feng \cite{feng2021generalized} connected mean-field generalized BSDEs with nonlocal PDEs; for the PDE theory of McKean--Vlasov equations and mean-field games, we refer to Cardaliaguet \cite{cardaliaguet2013notes} and Carmona and Delarue \cite{carmona2018probabilistic}. The mean reflection adds the integral constraint and the time-measure multiplier, and the associated PDE involves measure data, obstacle-type constraints in the spirit of variational inequalities, and the deterministic Skorokhod problem. Boccardo and Gallou\"{e}t \cite{boccardo1989nonlinear} studied nonlinear elliptic and parabolic equations involving measure data; Kinderlehrer and Stampacchia \cite{kinderlehrer2000introduction} developed the theory of variational inequalities; Kruk et al. \cite{kruk2007explicit} derived an explicit formula for the Skorokhod map on an interval. Delarue et al. \cite{delarue2015global} studied a networked integrate-and-fire model of McKean--Vlasov type, in which similar singular measures arise.

In this paper, we introduce a new class of generalized mean-field BSDEs subject to a mean-reflection constraint. Specifically, we study mean-reflected BSDEs where the driver depends on the joint law of $(Y, Z)$ and an additional integral driven by a continuous non-decreasing process $\kappa$ is present. The equation is given by:
\[
Y_t ={} \xi + \int_t^T f(s, Y_s, Z_s, \nu_s) \dif s + \int_t^T g(s, Y_s) \dif \kappa_s - \int_t^T Z_s \dif B_s + K_T - K_t,
\]
where $\nu_s$ denotes the law of $(Y_s, Z_s)$, subject to the mean reflection constraint $\mathbb{E}[Y_t] \ge{} u_t$ and the Skorokhod condition $\int_0^T (\mathbb{E}[Y_t] - u_t) \dif K_t ={} 0$. Compared to the existing literature, our model incorporates a general driver $f$ depending on the joint law of $(Y, Z)$ and an additional integral driven by $\kappa$ with a coefficient $g$ satisfying a strict monotonicity condition.

We establish the well-posedness of this generalized mean-reflected MF-BSDE. First, we derive a priori estimates and stability estimates for the solutions under Lipschitz and monotonicity assumptions. Second, we prove the uniqueness of the solution based on these stability estimates. Third, for existence, we use the penalization method. We construct a smooth approximation of the boundary $u$ and introduce a sequence of penalized equations without constraints. By establishing uniform estimates and the Cauchy property for the penalized solutions, we pass to the limit to obtain the existence of the solution to the original mean-reflected MF-BSDE. In addition, we study the Markovian case, where the equation reduces to a nonlocal parabolic PDE with a mean constraint. For reflected Brownian motion on an interval with a dissipative linear boundary driver, the mean reflection becomes an integral constraint on the PDE solution, and the reflection process becomes a deterministic time-measure source. We establish the well-posedness and the probabilistic representation of this problem, and we describe the regularity of the multiplier in terms of the regularity of the obstacle.

\section{Preliminaries}
\label{sec:preliminaries}

In this section, we introduce the mathematical notations and spaces used throughout this paper, and provide the mathematical formulation and the definition of the solution for the mean-reflected McKean--Vlasov BSDE. Finally, we state a series of basic assumptions upon which the results of this paper rely.

\subsection{Notations and Spaces}

Let $(\Omega,\mathcal{F},\mathbb{P})$ be a complete probability space on which a $d$-dimensional standard Brownian motion $B ={} (B_t)_{0\le{} t\le{} T}$ is defined. Denote by $\mathbb{F} ={} (\mathcal{F}_t)_{0\le{} t\le{} T}$ the natural filtration generated by this Brownian motion, augmented by all $\mathbb{P}$-null sets so that it satisfies the usual conditions (i.e., right-continuous and complete). Let $T>0$ be a given finite terminal time.

To formulate the equation and its solution, we introduce the following basic spaces of processes:
\begin{itemize}
\item $S^2$: the space of all $\mathbb{R}^n$-valued, continuous and $\mathbb{F}$-adapted stochastic processes $Y ={} (Y_t)_{0\le{} t\le{} T}$ such that
\[
\mathbb{E}\left[\sup_{0\le{} t\le{} T}|Y_t|^2\right]<\infty.
\]
\item $H^2$: the space of all $\mathbb{R}^{n\times d}$-valued, $\mathbb{F}$-predictable stochastic processes $Z ={} (Z_t)_{0\le{} t\le{} T}$ such that
\[
\mathbb{E}\left[\int_0^T|Z_t|^2\dif t\right]<\infty.
\]
\item $A_d^2$: the space of all deterministic, continuous and non-decreasing functions $K ={} (K_t)_{0\le{} t\le{} T}$ such that $K_0={}0$ and $K_T<\infty$.
\end{itemize}

Since our equation involves distribution dependence (i.e., McKean--Vlasov type dependence), we also need to define relevant measure spaces. Let $\mathcal{P}_2(\mathbb{R}^n\times\mathbb{R}^{n\times d})$ denote the space of all probability measures on $\mathbb{R}^n\times\mathbb{R}^{n\times d}$ with finite second moments. The space is naturally equipped with the 2-Wasserstein distance, defined as: for any $\nu_1,\nu_2\in\mathcal{P}_2(\mathbb{R}^n\times\mathbb{R}^{n\times d})$,
\[
W_2(\nu_1,\nu_2):={}\left(\inf_{\pi\in\Pi(\nu_1,\nu_2)}\int_{\mathbb{R}^n\times\mathbb{R}^{n\times d}}|x-y|^2\dif\pi(x,y)\right)^{1/2},
\]
where $\Pi(\nu_1,\nu_2)$ denotes the set of all joint probability measures with marginal distributions $\nu_1$ and $\nu_2$, respectively.

\section{Problem formulation and main result}
\label{sec:framework}

This section presents the mathematical formulation of the mean-reflected MF-GBSDE, the definition of its solution, and the basic assumptions used in this paper. Throughout, $\mathbb{F}$ denotes the usual augmentation of the natural filtration generated by the $d$-dimensional standard Brownian motion $B$.

\subsection{The equation and the definition of the solution}
\label{sec:equation}

We study the following mean-reflected MF-GBSDE on $[0,T]$:
\begin{equation}
\begin{cases}
Y_t = \xi + \int_t^T f(s, Y_s, Z_s, \nu_s) \dif s + \int_t^T g(s, Y_s) \dif \kappa_s + K_T - K_t - \int_t^T Z_s \dif B_s; \\[1ex]
\mathbb{E}[Y_t] \ge u_t, \qquad \int_0^T (\mathbb{E}[Y_t] - u_t) \dif K_t = 0,
\end{cases}
\label{eq:mf-gbsde}
\end{equation}
where $f: \Omega \times [0, T] \times \mathbb{R} \times \mathbb{R}^d \times \mathcal{P}_2(\mathbb{R} \times \mathbb{R}^d) \to \mathbb{R}$ is the distribution-dependent driver, $g: \Omega \times [0, T] \times \mathbb{R} \to \mathbb{R}$ is the boundary driver associated with the process $\kappa$, $\xi$ is the terminal value, $u$ is the deterministic obstacle, and $\kappa$ is a continuous non-decreasing process. The process $\nu_s$ denotes the joint law of the random vector $(Y_s, Z_s)$ at time $s$, that is, $\nu_s := \mathcal{L}(Y_s, Z_s)$. In \eqref{eq:mf-gbsde}, $K$ is called the reflection process; its role is to keep the expectation of $Y_t$ above the obstacle $u_t$. Moreover, the Skorokhod condition $\int_0^T (\mathbb{E}[Y_t] - u_t) \dif K_t = 0$ renders the action of the reflection process minimal, that is, $K$ increases only when $\mathbb{E}[Y_t] = u_t$.

\begin{remark}
\label{rem:c5-nu-convention}
The flow $\nu = (\nu_s)_{0 \le s \le T}$ of joint laws is chosen to be a measurable version. This modification does not affect the time or stochastic integrals; identities involving $\nu$ are understood in the $\dif t$-a.e. sense.
\end{remark}

\begin{definition}[Solution of the equation]
\label{def:solution}
A triple $(Y, Z, K) \in S^2 \times H^2 \times A_d^2$ is called a solution of equation \eqref{eq:mf-gbsde} if, for $\mathbb{P}$-almost every $\omega \in \Omega$, equation \eqref{eq:mf-gbsde} holds for all $0 \le t \le T$. Here and in the sequel, $S^2$ and $H^2$ are used with $n=1$; thus $Y$ is real-valued and $Z$ takes values in $\mathbb{R}^{1 \times d}$, which we identify with $\mathbb{R}^d$.

Since the mean constraint only involves the law of the solution, the solution is in general not unique if $K$ is allowed to be random. We therefore restrict $K$ to be deterministic, which is consistent with the definition of mean-reflected BSDEs in the literature.
\end{definition}

\subsection{Basic assumptions}
\label{assumption}

To establish the existence and uniqueness of the solution to equation \eqref{eq:mf-gbsde}, we impose the following assumptions on the driver $f$, the boundary driver $g$, the terminal condition $\xi$, the obstacle $u$, and the continuous non-decreasing process $\kappa$. For ease of reference, conditions \textbf{(A1)--(A5)} below are collectively referred to as assumption \textbf{(A)}. These conditions involve only unweighted square-integrability; in particular, no exponential moment condition on $\kappa_T$ or $\xi$ is required.

\medskip
\noindent\textbf{(A1)} \quad The map $f: \Omega \times [0, T] \times \mathbb{R} \times \mathbb{R}^d \times \mathcal{P}_2(\mathbb{R} \times \mathbb{R}^d) \to \mathbb{R}$ satisfies the following conditions:
\begin{enumerate}[label=(\alph*)]
  \item $f$ is jointly measurable with respect to $\mathcal{F} \otimes \mathcal{B}([0, T]) \otimes \mathcal{B}(\mathbb{R}) \otimes \mathcal{B}(\mathbb{R}^d) \otimes \mathcal{B}(\mathcal{P}_2(\mathbb{R} \times \mathbb{R}^d))$, and for every fixed $(y, z, \nu)$, the process $(f(t, y, z, \nu))_{0 \le t \le T}$ is progressively measurable; moreover,
  \[
  \mathbb{E}\left[ \int_0^T |f(t, 0, 0, \delta_{(0,0)})|^2 \dif t \right] < \infty,
  \]
  where $\delta_{(0,0)}$ denotes the Dirac measure concentrated at the origin of $\mathbb{R} \times \mathbb{R}^d$.
  \item There exists a constant $L_f \ge 0$ such that, for all $(t, y_i, z_i, \nu_i) \in [0, T] \times \mathbb{R} \times \mathbb{R}^d \times \mathcal{P}_2(\mathbb{R} \times \mathbb{R}^d)$ with $i = 1, 2$,
  \begin{equation}
  |f(t, y_1, z_1, \nu_1) - f(t, y_2, z_2, \nu_2)| \le L_f \left( |y_1 - y_2| + |z_1 - z_2| + W_2(\nu_1, \nu_2) \right).
  \label{eq:c5-A1-lipschitz}
  \end{equation}
\end{enumerate}

\medskip
\noindent\textbf{(A2)} \quad The map $g: \Omega \times [0, T] \times \mathbb{R} \to \mathbb{R}$ is jointly measurable with respect to $\mathcal{F} \otimes \mathcal{B}([0, T]) \otimes \mathcal{B}(\mathbb{R})$, the process $(g(t, y))_{0 \le t \le T}$ is progressively measurable for every fixed $y$, and for every $(\omega, t)$, the map $y \mapsto g(\omega, t, y)$ is continuous. There exists a constant $\beta < 0$ such that, for all $t \in [0, T]$ and $y_1, y_2 \in \mathbb{R}$,
\begin{equation}
(y_1 - y_2)(g(t, y_1) - g(t, y_2)) \le \beta (y_1 - y_2)^2.
\label{eq:c5-A2-dissipativity}
\end{equation}
Moreover, there exist a progressively measurable process $(\psi_t)_{0 \le t \le T}$ taking values in $[0, +\infty)$ and a constant $L_g > 0$ such that, for all $t \in [0, T]$ and $y \in \mathbb{R}$,
\begin{equation}
|g(t, y)| \le \psi_t + L_g |y| \quad \text{and} \quad \mathbb{E} \left[ \int_0^T \psi_t^2 \dif \kappa_t \right] < \infty.
\label{eq:c5-A2-growth}
\end{equation}
For convenience, we set $b := -\beta > 0$.

\medskip
\noindent\textbf{(A3)} \quad The terminal value $\xi$ is a square-integrable $\mathcal{F}_T$-measurable random variable and $\mathbb{E}[\xi] \ge u_T$.

\medskip
\noindent\textbf{(A4)} \quad The map $u: [0, T] \to \mathbb{R}$ is a deterministic continuous function.

\medskip
\noindent\textbf{(A5)} \quad The map $\kappa: [0, T] \to \mathbb{R}$ is an $\mathbb{F}$-adapted continuous non-decreasing process with $\kappa_0 = 0$ and
\[
a_T := \mathbb{E}[\kappa_T] < \infty.
\]

\subsection{Measures induced by $\kappa$ and densities of expected integrals}
\label{sec:kappa}

To handle expected integrals with respect to $\kappa$, we define the finite measure $q$ and the deterministic continuous non-decreasing function $a$
\[
\begin{aligned}
q(\dif\omega,\dif t)&:=\mathbb{P}(\dif\omega)\dif\kappa_t(\omega),&
\int H\dif q&:=\mathbb{E}\int_0^T H_t\dif\kappa_t,\\
a_t&:=\mathbb{E}\kappa_t,& A_t&:=t+a_t.
\end{aligned}
\]
Since $\kappa$ is continuous and $\kappa_T \in L^1$, the function $a$ is continuous and non-decreasing, so that $A$ is strictly increasing. For a deterministic integrable function $h$, we have $\mathbb{E}\int h \dif \kappa = \int h \dif a$. For a progressively measurable process $H$, set
\[
\|H\|_{L^2(q)}^2:=\mathbb{E}\int_0^T|H_t|^2\dif\kappa_t,\qquad
L^2(q):=\left\{H \text{ progressively measurable}: \|H\|_{L^2(q)}<\infty\right\}.
\]
In the subsequent a priori estimates and error analysis, we will frequently deal with expected integrals with respect to $\kappa$, of the form $\mathbb{E}\int_0^T H_t \dif \kappa_t$. To embed such integrals into the framework of deterministic integrals, we need the following density representation.

\begin{lemma}[Radon--Nikodym density of the expected integral]
\label{lem:c5-density}
Let $H$ be a progressively measurable process with $H \in L^2(q)$. Then there exists an $\bar H \in L^2(\dif a)$, unique up to $L^2(\dif a)$-equivalence, such that for every Borel set $I \subseteq [0, T]$,
\begin{equation*}
\mathbb{E}\int_I H_t \dif \kappa_t = \int_I \bar H_t \dif a_t,
\end{equation*}
and moreover
\begin{equation*}
\int_0^T |\bar H_t|^2 \dif a_t \le \mathbb{E}\int_0^T |H_t|^2 \dif \kappa_t.
\end{equation*}
In general, $\bar H_t \ne \mathbb{E}H_t$.
\end{lemma}

\begin{proof}
For deterministic simple functions $\varphi$, the Cauchy--Schwarz inequality gives
\[
\left|\mathbb{E}\int_0^T \varphi_t H_t \dif \kappa_t\right| \le \|\varphi\|_{L^2(\dif a)} \|H\|_{L^2(q)}.
\]
This functional extends continuously to $L^2(\dif a)$. The Riesz representation theorem yields the unique equivalence class $\bar H$, whose norm is bounded by $\|H\|_{L^2(q)}$. Taking $\varphi = \mathbf{1}_I$ gives the representation formula.
\end{proof}

\begin{remark}
\label{rem:c5-density}
Lemma \ref{lem:c5-density} expresses expected integrals with respect to $\kappa$ as integrals with respect to the deterministic measure $\dif a$. To unify integrals with respect to $\dif t$ and $\dif \kappa$, define
\[
\theta_t:=\frac{\dif t}{\dif A}(t),\qquad
\eta_t:=\frac{\dif a}{\dif A}(t),\qquad
0 \le \theta_t, \eta_t \le 1, \quad \theta_t + \eta_t = 1
\]
for $\dif A$-almost every $t$.
\end{remark}

The above density representation is used repeatedly in the a priori estimates; the following two technical lemmas are used in the subsequent limiting procedures.

\begin{lemma}[Continuity of the Nemytskii operator]
\label{lem:nemytskii}
Let $X_j \to X$ in $L^2(q)$. Let $g: \Omega \times [0, T] \times \mathbb{R} \to \mathbb{R}$ be jointly measurable with respect to $\mathcal{F} \otimes \mathcal{B}([0, T]) \otimes \mathcal{B}(\mathbb{R})$, continuous in $y$ for every $(\omega, t)$, and assume that there exist a constant $L_g > 0$ and a progressively measurable process $\psi$ such that
\[
|g(t, y)| \le \psi_t + L_g |y|, \qquad \psi \in L^2(q).
\]
Then
\[
g(\cdot, X_j) \longrightarrow g(\cdot, X) \quad \text{in } L^2(q).
\]
\end{lemma}

\begin{proof}
Since $X_j \to X$ in $L^2(q)$, every subsequence contains a further subsequence converging to $X$ $q$-almost everywhere. By the continuity of $y \mapsto g(t, y)$, along this further subsequence we have $g(\cdot, X_j) \to g(\cdot, X)$ $q$-almost everywhere. By the subsequence criterion for convergence in measure, $g(\cdot, X_j) \to g(\cdot, X)$ in measure with respect to $q$.

On the other hand, the growth condition \eqref{eq:c5-A2-growth} gives
\[
|g(t, X_j) - g(t, X)|^2 \le 8\psi_t^2 + 4L_g^2 \bigl( |X_j|^2 + |X|^2 \bigr) \le 8\psi_t^2 + 8L_g^2 |X_j - X|^2 + 12 L_g^2 |X|^2.
\]
Since $X_j \to X$ in $L^2(q)$, we have $\|\,|X_j - X|^2 \|_{L^1(q)} = \|X_j - X\|_{L^2(q)}^2 \to 0$, so the sequence $\{|X_j - X|^2\}$ is uniformly integrable in $L^1(q)$ (a convergent sequence in $L^1$ is necessarily uniformly integrable); since $|X|^2, \psi^2 \in L^1(q)$, the display above shows that $\{|g(\cdot, X_j) - g(\cdot, X)|^2\}$ is uniformly integrable in $L^1(q)$. By Vitali's theorem, $g(\cdot, X_j) \to g(\cdot, X)$ in $L^2(q)$.
\end{proof}

\begin{lemma}[Identification of two strong limits]
\label{lem:limit-identification}
Let $(Y^j)_{j \ge 1}$ be a sequence of continuous adapted processes that is Cauchy both in $S^2$ and in $L^2(q)$. Then there exists $Y \in S^2 \cap L^2(q)$ such that
\[
Y^j \longrightarrow Y \quad \text{in } S^2,
\qquad
Y^j \longrightarrow Y \quad \text{in } L^2(q).
\]
In particular, $S^2 \cap L^2(q)$ is complete under the norm $\|\cdot\|_{S^2} + \|\cdot\|_{L^2(q)}$.
\end{lemma}

\begin{proof}
By the completeness of $S^2$ and $L^2(q)$, there exist $Y \in S^2$ and $\widehat{Y} \in L^2(q)$ such that
\[
Y^j \longrightarrow Y \quad \text{in } S^2,
\qquad
Y^j \longrightarrow \widehat{Y} \quad \text{in } L^2(q)
\qquad (j \to \infty).
\]
We prove that $Y = \widehat{Y}$ $q$-almost everywhere. By the $S^2$ convergence, we can choose a subsequence $(j_\ell)$ such that
\[
\sum_{\ell = 1}^\infty \mathbb{E}\left[ \sup_{0 \le t \le T} |Y^{j_\ell}_t - Y_t|^2 \right] < \infty,
\]
so that along this subsequence, $\sup_{0 \le t \le T} |Y^{j_\ell}_t - Y_t| \to 0$ almost surely. Let $N \subset \Omega$ be the corresponding null set. Since $\kappa_T \in L^1$,
\[
q(N \times [0, T]) = \mathbb{E}\left[ \mathbf{1}_N \kappa_T \right] = 0,
\]
so $Y^{j_\ell} \to Y$ $q$-almost everywhere, hence in measure with respect to $q$. On the other hand, the $L^2(q)$ convergence implies $Y^{j_\ell} \to \widehat{Y}$ in measure with respect to $q$; by the uniqueness of limits in measure, $Y = \widehat{Y}$ $q$-almost everywhere, and in particular $Y \in L^2(q)$.

It remains to prove that $Y^j \to Y$ in $L^2(q)$. Let $\varepsilon > 0$. Since $(Y^j)$ is Cauchy in $L^2(q)$, there exists $j_0$ such that $\|Y^j - Y^{j'}\|_{L^2(q)} \le \varepsilon$ whenever $j, j' \ge j_0$. Along the subsequence above, choose $j_\ell \ge j_0$ large enough that $\|Y^{j_\ell} - Y\|_{L^2(q)} \le \varepsilon$; then for $j \ge j_0$,
\[
\|Y^j - Y\|_{L^2(q)} \le \|Y^j - Y^{j_\ell}\|_{L^2(q)} + \|Y^{j_\ell} - Y\|_{L^2(q)} \le 2\varepsilon.
\]
The lemma is proved.
\end{proof}

\subsection{Main result}
\label{sec:main-result}

With the notation above, the well-posedness of equation \eqref{eq:mf-gbsde} and norm estimates for its solution are the main objective of this paper. The subsequent sections establish a priori estimates, stability and uniqueness, and construct the solution by the penalization method, leading to the following well-posedness theorem. Here and below, $\|Y\|_{S^2}^2 := \mathbb{E}[\sup_{0 \le t \le T} |Y_t|^2]$ and $\|Z\|_{H^2}^2 := \mathbb{E}[\int_0^T |Z_t|^2 \dif t]$.

\begin{theorem}[Existence and uniqueness]
\label{thm:existence_uniqueness}
Under assumption \textbf{(A)}, the mean-reflected MF-GBSDE \eqref{eq:mf-gbsde} has a unique solution $(Y, Z, K)$ in the space $S^2 \times H^2 \times A_d^2$, and $Y \in L^2(q)$. Let $M := \sup_{0 \le t \le T} u_t^+$. Then there exists a constant $C > 0$, depending only on $L_f, L_g, b, T, a_T$, such that
\begin{equation}
\begin{aligned}
&\|Y\|_{S^2}^2 + \|Z\|_{H^2}^2 + \|Y\|_{L^2(q)}^2 + K_T^2 \\
\le {}& C \left( \mathbb{E}|\xi|^2 + \mathbb{E} \int_0^T |f(t, 0, 0, \delta_{(0,0)})|^2 \dif t + \mathbb{E} \int_0^T \psi_t^2 \dif \kappa_t + M^2 (1 + T + a_T) \right).
\end{aligned}
\label{eq:c5-main-estimate}
\end{equation}
\end{theorem}

The proof of Theorem \ref{thm:existence_uniqueness} proceeds as follows. We first establish a priori estimates and difference energy estimates, from which uniqueness follows. For existence, we penalize a smooth obstacle with respect to the deterministic function $A$, where an auxiliary generalized BSDE provides solvability for fixed parameters, and uniform bounds together with the negative-part error estimate yield the convergence of the penalized solutions. Finally, we pass to the limit successively in the penalization parameter and in the obstacle approximation parameter, and verify the Skorokhod condition through the expectation representation of the reflection and the convergence of the Stieltjes integrals.

\section{A priori estimates}
\label{sec:priori}

This section establishes a priori estimates by exploiting the dissipativity \textbf{(A2)} of the boundary driver and the determinism of the reflection. We first prove that any solution of the dynamics automatically belongs to $L^2(q)$, which guarantees the integrability of the expected integrals with respect to $\kappa$.

\begin{lemma}[Square-integrability of the solution with respect to $\kappa$]
\label{lem:c5-Y-L2q}
Under assumption \textbf{(A)}, let $(Y,Z,K)\in S^2\times H^2\times A_d^2$ satisfy the dynamics
\begin{equation}
\begin{aligned}
Y_t={}&\xi+\int_t^T f(s,Y_s,Z_s,\nu_s)\dif s+\int_t^T g(s,Y_s)\dif\kappa_s\\
&+K_T-K_t-\int_t^T Z_s\dif B_s,\qquad 0\le t\le T,
\end{aligned}
\label{eq:c5-priori-equation}
\end{equation}
where $\nu_s:=\mathcal{L}(Y_s,Z_s)$. Then $\mathbb{E}\int_0^T|Y_t|^2\dif\kappa_t<\infty$.
\end{lemma}

\begin{proof}
By path continuity and $\kappa_T<\infty$ \textbf{(A5)}, and by the growth condition \textbf{(A2)}, the pathwise integrals $\int|Y|^2\dif\kappa$, $\int|Yg(\cdot,Y)|\dif\kappa$ and $\int|Y|^2|Z|^2\dif t$ are finite. For $j\ge1$, let $\sigma_j$ be the stopping time that truncates the sum of these three cumulative quantities at level $j$; then $\sigma_j\uparrow T$, and the stopped It\^o stochastic integral is a martingale.

Set $F_s=f(s,Y_s,Z_s,\nu_s)$. Since $Y\in S^2$ and $Z\in H^2$, we have $F\in L^2(\mathbb{P}\otimes\dif t)$. Applying It\^o's formula to $|Y|^2$ on $[0,\sigma_j]$ and using $2Yg(s,Y)\le-b|Y|^2+b^{-1}\psi_s^2$, we obtain
\[
\begin{aligned}
b\mathbb{E}\int_0^{\sigma_j}|Y_s|^2\dif\kappa_s\le{}&
\mathbb{E}\sup_t|Y_t|^2+2\mathbb{E}\int_0^T|Y_sF_s|\dif s\\
&+b^{-1}\mathbb{E}\int_0^T\psi_s^2\dif\kappa_s
+2K_T\mathbb{E}\sup_t|Y_t|.
\end{aligned}
\]
The right-hand side is finite and independent of $j$. The monotone convergence theorem gives the conclusion.
\end{proof}

With Lemma \ref{lem:c5-Y-L2q} guaranteeing the integrability of the expected boundary integrals, we further establish energy estimates for the solution. The following a priori estimate controls the norm of the solution by the terminal value, the inhomogeneous term of the driver, the growth process $\psi$ of the boundary driver, and the terminal value of the reflection process.

\begin{proposition}[A priori estimate]
\label{prop:c5-priori}
Under assumption \textbf{(A)}, let $(Y,Z,K)\in S^2\times H^2\times A_d^2$ satisfy equation \eqref{eq:c5-priori-equation}. Then there exists a constant $C>0$, depending only on $L_f,b,T$, such that
\begin{equation}
\begin{aligned}
&\mathbb{E}\sup_{t\le T}|Y_t|^2+\mathbb{E}\int_0^T|Z_s|^2\dif s
+b\mathbb{E}\int_0^T|Y_s|^2\dif\kappa_s\\
\le {}& C\left(\|(\xi, f^0)\|^2
+\mathbb{E}\int_0^T\psi_s^2\dif\kappa_s+K_T^2\right),
\end{aligned}
\label{eq:c5-priori-estimate}
\end{equation}
where $f_s^0:=f(s,0,0,\delta_{(0,0)})$ and $\|(\xi,f^0)\|^2:=\mathbb{E}[|\xi|^2+\int_0^T|f_s^0|^2\dif s]$.
\end{proposition}

\begin{proof}
By Lemma \ref{lem:c5-Y-L2q}, all expected boundary integrals are integrable. Set $F_s=f(s,Y_s,Z_s,\nu_s)$. The Lipschitz condition \eqref{eq:c5-A1-lipschitz} and Young's inequality give
\begin{equation}
2|Y_s||F_s|\le c_f|Y_s|^2+\tfrac14|Z_s|^2+|f_s^0|^2
+\tfrac14(\mathbb{E}|Y_s|^2+\mathbb{E}|Z_s|^2),
\label{eq:c5-priori-f}
\end{equation}
while dissipativity \textbf{(A2)} gives
\begin{equation}
2Y_sg(s,Y_s)\le-b|Y_s|^2+b^{-1}\psi_s^2.
\label{eq:c5-priori-g}
\end{equation}
Choose $\alpha$ sufficiently large and set
\[
\begin{aligned}
S&=\mathbb{E}\sup_{t\le T}e^{\alpha t}|Y_t|^2,\\
J&=\mathbb{E}\int_0^T e^{\alpha s}(|Y_s|^2+|Z_s|^2)\dif s
+b\mathbb{E}\int_0^T e^{\alpha s}|Y_s|^2\dif\kappa_s,\\
\Theta_\alpha&=\|(\xi, f^0)\|_\alpha^2
+b^{-1}\mathbb{E}\int_0^T e^{\alpha s}\psi_s^2\dif\kappa_s,
\end{aligned}
\]
where
\[
\|(\xi, f^0)\|_\alpha^2
:=\mathbb{E}\left[e^{\alpha T}|\xi|^2+\int_0^T e^{\alpha s}|f_s^0|^2\dif s\right]
\]
is the weighted data norm. Applying the weighted It\^o formula and taking expectations (the martingale term belongs to $H^1$ because $Y\in S^2$ and $Z\in H^2$), absorbing the driver terms into the left-hand side, and controlling the reflection term by $2\mathbb{E}\int e^{\alpha s}|Y_s|\dif K_s\le\varepsilon S+C_\varepsilon K_T^2$, we obtain
\begin{equation}
J\le C\Theta_\alpha+\varepsilon S+C_\varepsilon K_T^2.
\label{eq:c5-priori-energy-first}
\end{equation}
For the pathwise estimate, we first use inequality \eqref{eq:c5-priori-g} at each fixed time and then take the supremum, so that only $\psi$ remains in the boundary term; the driver is controlled by inequality \eqref{eq:c5-priori-f}, the reflection term by Young's inequality, and the martingale term by the BDG inequality:
\[
\mathbb{E}\sup_t\left|\int_t^T e^{\alpha s}Y_sZ_s\dif B_s\right|
\le\varepsilon S+C_\varepsilon J.
\]
After absorption we obtain
\[
S\le C(\Theta_\alpha+J+K_T^2).
\]
Choosing $\varepsilon$ sufficiently small in inequality \eqref{eq:c5-priori-energy-first} and combining the two inequalities yields $S+J\le C(\Theta_\alpha+K_T^2)$. Removing the deterministic exponential weight gives the conclusion.
\end{proof}

\begin{remark}
\label{rem:c5-priori}
The constant in Proposition \ref{prop:c5-priori} does not depend on $L_g$: the pathwise estimate uses the dissipativity \eqref{eq:c5-priori-g} to control the boundary term. Lemma \ref{lem:c5-Y-L2q} also guarantees that any solution of the dynamics automatically belongs to $L^2(q)$.
\end{remark}

\section{Stability}
\label{sec:stability}

The a priori estimates of the previous section show that the norm of the solution is controlled by the system data. In this section, we study the effect of perturbations of the data (terminal condition, driver and boundary driver) on the solution. Compared with the classical case, since the equation contains both an additional integral driven by $\kappa$ and the mean reflection, the proof of stability requires retaining the dissipative term of the boundary driver and handling expected integrals with respect to $\kappa$ correctly. The proposition below shows that, for two solutions with the same $\kappa$ and the same obstacle, the supremum in time $\sup_{0 \le t \le T} \mathbb{E}|\delta Y_t|^2$, the $H^2$ norm and the $\kappa$-weighted norm of the difference are controlled by the terminal error and the differences of the drivers, with a constant depending only on $L_f, L_g, b, T$.

\begin{proposition}[Energy stability]
\label{lem:stability}
Under assumption \textbf{(A)}, let $(Y^i, Z^i, K^i)$ be solutions of equation \eqref{eq:mf-gbsde} corresponding to terminal conditions $\xi^i$ and drivers $(f^i, g^i)$, $i = 1, 2$, where $g^1, g^2$ have the same dissipativity constant $b > 0$, $f^1, f^2$ have the same Lipschitz constant $L_f$, and the two solutions share the same $\kappa$ and obstacle $u$. Set
\[
\delta Y_t := Y_t^1 - Y_t^2, \quad \delta Z_t := Z_t^1 - Z_t^2, \quad \delta K_t := K_t^1 - K_t^2, \quad \delta \xi := \xi^1 - \xi^2,
\]
and denote the differences of the drivers by
\[
\delta f_s := f^1(s, Y_s^2, Z_s^2, \nu_s^2) - f^2(s, Y_s^2, Z_s^2, \nu_s^2), \qquad
\delta g_s := g^1(s, Y_s^2) - g^2(s, Y_s^2).
\]
Then there exists a constant $C > 0$, depending only on $L_f, L_g, b, T$, such that
\begin{equation}
\begin{aligned}
&\sup_{0 \le t \le T} \mathbb{E}|\delta Y_t|^2
+ \mathbb{E} \int_0^T |\delta Z_s|^2 \dif s
+ b \mathbb{E} \int_0^T |\delta Y_s|^2 \dif \kappa_s \\
\le {}&  C \left(
\mathbb{E}|\delta \xi|^2
+ \mathbb{E} \int_0^T |\delta f_s|^2 \dif s
+ \frac{1}{b} \mathbb{E} \int_0^T |\delta g_s|^2 \dif \kappa_s
\right).
\end{aligned}
\label{eq:c5-stability}
\end{equation}
\end{proposition}

\begin{proof}
Apply It\^o's formula to $e^{\alpha t} |\delta Y_t|^2$ on $[t, T]$ and take expectations, where $\alpha > 0$ is to be determined. Since $\delta Y \in S^2$ and $\delta Z \in H^2$, the BDG inequality shows that the local martingale $\int_0^\cdot e^{\alpha s} \delta Y_s \delta Z_s \dif B_s$ belongs to $H^1$, hence is a uniformly integrable martingale with zero expectation. Set
\[
\Delta f_s := f^1(s, Y_s^1, Z_s^1, \nu_s^1) - f^1(s, Y_s^2, Z_s^2, \nu_s^2);
\]
then
\begin{align}
&\mathbb{E}\left[ e^{\alpha t} |\delta Y_t|^2 \right]
+ \mathbb{E} \int_t^T e^{\alpha s} |\delta Z_s|^2 \dif s
+ \alpha \mathbb{E} \int_t^T e^{\alpha s} |\delta Y_s|^2 \dif s \notag \\
= {}& \mathbb{E}\left[ e^{\alpha T} |\delta \xi|^2 \right]
+ 2 \mathbb{E} \int_t^T e^{\alpha s} \delta Y_s \Delta f_s \dif s
+ 2 \mathbb{E} \int_t^T e^{\alpha s} \delta Y_s \delta f_s \dif s \notag \\
&\qquad\quad + 2 \mathbb{E} \int_t^T e^{\alpha s} \delta Y_s \left( g^1(s, Y_s^1) - g^2(s, Y_s^2) \right) \dif \kappa_s
+ 2 \mathbb{E} \int_t^T e^{\alpha s} \delta Y_s \dif \delta K_s.
\label{eq:c5-stab-ito}
\end{align}

\medskip
\noindent\textbf{Step 1: Estimate of the driver difference.}
By the Lipschitz condition \textbf{(A1)} and $W_2(\nu_s^1, \nu_s^2)^2 \le \mathbb{E}|\delta Y_s|^2 + \mathbb{E}|\delta Z_s|^2$, we have pointwise
\begin{align}
&2 \delta Y_s \left( \Delta f_s + \delta f_s \right) \notag \\
\le {}&  \left( 2 L_f + 8 L_f^2 + 1 \right) |\delta Y_s|^2
+ \frac{1}{4} |\delta Z_s|^2
+ \frac{1}{4} \mathbb{E}|\delta Y_s|^2
+ \frac{1}{4} \mathbb{E}|\delta Z_s|^2
+ |\delta f_s|^2.
\label{eq:c5-stab-f}
\end{align}

\medskip
\noindent\textbf{Step 2: Estimate of the boundary driver difference.}
Using first the monotonicity \eqref{eq:c5-A2-dissipativity} of each driver and then separating the difference of the drivers themselves, we get
\begin{align*}
2 \delta Y_s \left( g^1(s, Y_s^1) - g^2(s, Y_s^2) \right)
& = 2 \delta Y_s \left( g^1(s, Y_s^1) - g^1(s, Y_s^2) \right) + 2 \delta Y_s \delta g_s \\
& \le -2 b |\delta Y_s|^2 + 2 |\delta Y_s| |\delta g_s|.
\end{align*}
By Young's inequality $2 |\delta Y_s| |\delta g_s| \le b |\delta Y_s|^2 + b^{-1} |\delta g_s|^2$,
\[
2 \delta Y_s \left( g^1(s, Y_s^1) - g^2(s, Y_s^2) \right)
\le -b |\delta Y_s|^2 + \frac{1}{b} |\delta g_s|^2.
\]
Here the term $-b |\delta Y_s|^2 \dif \kappa_s$ is kept on the left-hand side; it is exactly the dissipative term needed later to control the $\kappa$-weighted norm of $\delta Y$.

\medskip
\noindent\textbf{Step 3: Estimate of the reflection difference.}
Since $K^1, K^2$ are deterministic and both solutions satisfy the same constraint $\mathbb{E}[Y_s^i] \ge u_s$ and the corresponding Skorokhod conditions, expanding the cross terms gives
\begin{align*}
2 \mathbb{E} \int_t^T e^{\alpha s} \delta Y_s \dif \delta K_s
&= 2 \int_t^T e^{\alpha s} \left( \mathbb{E}[Y_s^1] - u_s \right) \dif K_s^1
- 2 \int_t^T e^{\alpha s} \left( \mathbb{E}[Y_s^1] - u_s \right) \dif K_s^2 \notag \\
&\quad - 2 \int_t^T e^{\alpha s} \left( \mathbb{E}[Y_s^2] - u_s \right) \dif K_s^1 \notag \\
&\quad + 2 \int_t^T e^{\alpha s} \left( \mathbb{E}[Y_s^2] - u_s \right) \dif K_s^2.
\end{align*}
The first and fourth terms vanish by the Skorokhod conditions; the integrand in the second term is non-negative while $\dif K_s^2 \ge 0$, so this term is non-positive; likewise for the third term. Hence
\begin{equation}
2 \mathbb{E} \int_t^T e^{\alpha s} \delta Y_s \dif \delta K_s \le 0.
\label{eq:c5-stab-k}
\end{equation}

\medskip
\noindent\textbf{Step 4: Absorption and conclusion.}
Substitute inequalities \eqref{eq:c5-stab-f}--\eqref{eq:c5-stab-k} into equality \eqref{eq:c5-stab-ito}. As in the a priori estimate, use Fubini's theorem to write the deterministic expectation terms as integrals with respect to $\dif s$, and take
\[
\alpha \ge 2 L_f + 8 L_f^2 + \frac{5}{4}.
\]
Then the terms involving $|\delta Y_s|^2$ on the right-hand side are absorbed by the left-hand side, while the right-hand side term $\frac{1}{2} \mathbb{E}\int e^{\alpha s} |\delta Z_s|^2 \dif s$ merges with the left-hand side leaving one half, so
\[
\begin{aligned}
&\mathbb{E}\left[ e^{\alpha t} |\delta Y_t|^2 \right]
+ \frac{1}{2} \mathbb{E} \int_t^T e^{\alpha s} |\delta Z_s|^2 \dif s
+ b \mathbb{E} \int_t^T e^{\alpha s} |\delta Y_s|^2 \dif \kappa_s \\
\le {}&  \mathbb{E}\left[ e^{\alpha T} |\delta \xi|^2 \right]
+ \mathbb{E} \int_t^T e^{\alpha s} |\delta f_s|^2 \dif s
+ \frac{1}{b} \mathbb{E} \int_t^T e^{\alpha s} |\delta g_s|^2 \dif \kappa_s.
\end{aligned}
\]
Taking the supremum over $t \in [0, T]$ and using $1 \le e^{\alpha t} \le e^{\alpha T}$ to replace the weighted norms by unweighted ones yields the stability estimate \eqref{eq:c5-stability}.
\end{proof}

\begin{remark}[$S^2$ stability]
\label{rem:c5-S2-stability}
Under the assumptions of Proposition \ref{lem:stability}, set
\[
R := \mathbb{E}|\delta \xi|^2
+ \mathbb{E} \int_0^T |\delta f_s|^2 \dif s
+ \frac{1}{b} \mathbb{E} \int_0^T |\delta g_s|^2 \dif \kappa_s,
\qquad
\Lambda := K_T^1 + K_T^2.
\]
Since $K^1, K^2$ are deterministic, $\Lambda$ is a finite deterministic number. Keeping the dissipative estimate of the boundary driver in the difference It\^o formula and then taking the stochastic supremum, the reflection term satisfies
\begin{align*}
\mathbb{E}\left[ \sup_{0 \le t \le T} \left| \int_t^T e^{\alpha s} \delta Y_s \dif \delta K_s \right| \right]
&\le e^{\alpha T} \int_0^T \mathbb{E}|\delta Y_s| \dif (K_s^1 + K_s^2) \\
&\le e^{\alpha T} \Lambda \left( \sup_{0 \le s \le T} \mathbb{E}|\delta Y_s|^2 \right)^{1/2},
\end{align*}
while the driver terms and the martingale term are controlled by $C R$ and $\frac{1}{4} \mathbb{E}[\sup_{0 \le t \le T} e^{\alpha t} |\delta Y_t|^2] + C \mathbb{E}\int_0^T |\delta Z_s|^2 \dif s$, respectively. Combining this with the stability estimate \eqref{eq:c5-stability} and absorbing the stochastic supremum on the left-hand side, we obtain
\[
\mathbb{E}\left[ \sup_{0 \le t \le T} |\delta Y_t|^2 \right]
+ \mathbb{E} \int_0^T |\delta Z_s|^2 \dif s
\le C \left( R + \Lambda \sqrt{R} \right),
\]
where $C$ depends only on $L_f, L_g, b, T$.

More generally, the above difference calculation uses only the dynamics. If the reflection pairing satisfies
\[
2\mathbb{E}\int_t^T e^{\alpha s}\delta Y_s\dif\delta K_s\le E_K,\qquad E_K\ge0,
\]
for all $t$, then the same calculation gives
\begin{equation}
\begin{aligned}
&\sup_t\mathbb{E}|\delta Y_t|^2+\|\delta Z\|_{H^2}^2+b\|\delta Y\|_{L^2(q)}^2\le C(R+E_K),\\
&\|\delta Y\|_{S^2}^2+\|\delta Z\|_{H^2}^2\le C\bigl(R+E_K+\Lambda\sqrt{R+E_K}\bigr).
\end{aligned}
\label{eq:c5-difference-budget}
\end{equation}
The second estimate again controls the pathwise reflection term by its total variation.
\end{remark}

\section{Uniqueness}
\label{sec:uniqueness}

Using the stability of the previous section, we obtain the uniqueness of the solution to equation \eqref{eq:mf-gbsde}. We first write down the expectation representation of the reflection process $K$.

\begin{lemma}[Expectation representation of the reflection process]
\label{lem:c5-K-formula}
Under assumption \textbf{(A)}, if $(Y,Z,K)\in S^2\times H^2\times A_d^2$ satisfies the dynamics \eqref{eq:c5-priori-equation}, then for all $t\in[0,T]$,
\begin{equation}
K_t = \mathbb{E}[Y_0] - \mathbb{E}[Y_t]
- \int_0^t \mathbb{E}\left[ f(s, Y_s, Z_s, \nu_s) \right] \dif s
- \mathbb{E} \int_0^t g(s, Y_s) \dif \kappa_s.
\label{eq:c5-K-expectation}
\end{equation}
\end{lemma}

\begin{proof}
Rewrite equation \eqref{eq:mf-gbsde} as
\[
K_t = Y_0 - Y_t - \int_0^t f(s, Y_s, Z_s, \nu_s) \dif s
- \int_0^t g(s, Y_s) \dif \kappa_s + \int_0^t Z_s \dif B_s.
\]
Take expectations on both sides. Since $Z \in H^2$, the stochastic integral $\int_0^\cdot Z_s \dif B_s$ has zero expectation; since $\mathbb{E}\int_0^T |f| \dif s < \infty$ and $\mathbb{E}\int_0^T |g| \dif \kappa_s < \infty$ (the latter follows from the growth condition \eqref{eq:c5-A2-growth} and the a priori estimate \eqref{eq:c5-priori-estimate}), Fubini's theorem gives
\[
\mathbb{E} \int_0^t f(s, Y_s, Z_s, \nu_s) \dif s
= \int_0^t \mathbb{E}\left[ f(s, Y_s, Z_s, \nu_s) \right] \dif s.
\]
The integral with respect to $\kappa$ remains as $\mathbb{E}\int_0^t g(s, Y_s) \dif \kappa_s$. Combining these facts gives the expectation representation \eqref{eq:c5-K-expectation}.
\end{proof}

Combining the stability estimate with the expectation representation of the reflection process, we can discuss the uniqueness of the solution. The following proposition shows that, under assumption \textbf{(A)}, the solution of equation \eqref{eq:mf-gbsde} is unique.

\begin{proposition}[Uniqueness of the solution]
\label{prop:uniqueness}
Under assumption \textbf{(A)}, the solution of equation \eqref{eq:mf-gbsde} in the space $S^2 \times H^2 \times A_d^2$ is unique.
\end{proposition}

\begin{proof}
Let $(Y^1, Z^1, K^1)$ and $(Y^2, Z^2, K^2)$ be two solutions of equation \eqref{eq:mf-gbsde}. In the stability proposition \ref{lem:stability}, take the two solutions to have the same terminal condition and the same drivers, that is, $\delta \xi = 0$, $\delta f = 0$, $\delta g = 0$. The stability estimate \eqref{eq:c5-stability} gives
\[
\sup_{0 \le t \le T} \mathbb{E}|\delta Y_t|^2
+ \mathbb{E} \int_0^T |\delta Z_s|^2 \dif s
+ b \mathbb{E} \int_0^T |\delta Y_s|^2 \dif \kappa_s \le 0.
\]
Hence $\mathbb{E}|\delta Y_t|^2 = 0$ for all $t \in [0, T]$, and $\mathbb{E}\int_0^T |\delta Z_s|^2 \dif s = 0$. By the path continuity of $Y$, there exists $\Omega_0$ with $\mathbb{P}(\Omega_0) = 1$ such that $Y_t^1(\omega) = Y_t^2(\omega)$ for all $\omega \in \Omega_0$ and all $t \in [0, T]$; moreover, $\delta Z = 0$ holds $\mathbb{P} \otimes \dif s$-almost everywhere. It follows that $\nu_s^1 = \nu_s^2$ for $\dif s$-almost every $s$.

For the reflection process, Lemma \ref{lem:c5-K-formula} gives, for $i = 1, 2$,
\[
K_t^i = \mathbb{E}[Y_0^i] - \mathbb{E}[Y_t^i]
- \int_0^t \mathbb{E}\left[ f(s, Y_s^i, Z_s^i, \nu_s^i) \right] \dif s
- \mathbb{E} \int_0^t g(s, Y_s^i) \dif \kappa_s.
\]
The right-hand side is expressed only through $(Y^i, Z^i, \nu^i)$ and the common data, and since $Y^1 = Y^2$ and $Z^1 = Z^2$ $\mathbb{P} \otimes \dif s$-almost everywhere and $\nu^1 = \nu^2$ $\dif s$-almost everywhere by what has been proved above, we obtain $K_t^1 = K_t^2$ for all $t \in [0, T]$. The proposition is proved.
\end{proof}

\section{Well-posedness of an auxiliary generalized BSDE}
\label{sec:auxiliary_gbsde}

To handle the distribution dependence and the deterministic mean feedback, this section establishes the well-posedness of an auxiliary generalized BSDE without the distribution term. Since only $\mathbb{E}\kappa_T<\infty$ is assumed in our framework, we need to truncate the data and the process $\kappa$, and then remove the truncations by means of dissipative estimates.

\begin{lemma}[Well-posedness of the auxiliary generalized BSDE]
\label{lem:c5-aux}
Let $F: \Omega \times [0, T] \times \mathbb{R} \times \mathbb{R}^d \to \mathbb{R}$ be jointly measurable with respect to $\mathcal{F} \otimes \mathcal{B}([0, T]) \otimes \mathcal{B}(\mathbb{R}) \otimes \mathcal{B}(\mathbb{R}^d)$ and progressively measurable in $t$ for every fixed $(y, z)$, and assume that there exists a constant $L \ge 0$ such that
\[
|F(t, y_1, z_1) - F(t, y_2, z_2)| \le L \left( |y_1 - y_2| + |z_1 - z_2| \right)
\]
for all $t \in [0, T]$, $y_1, y_2 \in \mathbb{R}$ and $z_1, z_2 \in \mathbb{R}^d$, and that
\[
\mathbb{E} \int_0^T |F(t, 0, 0)|^2 \dif t < \infty.
\]
Let $\eta \in L^2(\mathcal{F}_T)$, let $g$ and $\kappa$ satisfy assumptions \textbf{(A2)} and \textbf{(A5)}, and let $P: [0, T] \to \mathbb{R}$ be a deterministic continuous function of finite variation with $P_0 = 0$. Then the equation
\begin{equation}
Y_t = \eta + \int_t^T F(s, Y_s, Z_s) \dif s
+ \int_t^T g(s, Y_s) \dif \kappa_s
+ P_T - P_t - \int_t^T Z_s \dif B_s
\label{eq:c5-aux-equation}
\end{equation}
has a unique solution in $(S^2 \cap L^2(q)) \times H^2$.
\end{lemma}

\begin{proof}
We first make the change of variables $\widetilde Y = Y + P$ to remove the deterministic drift. The new terminal value is $\eta + P_T$, and the new drivers are $F(t, y - P_t, z)$ and $g(t, y - P_t)$. The Lipschitz \textbf{(A1)} and dissipativity \textbf{(A2)} constants are unchanged, and the driver at the origin $F(t, -P_t, 0)$ is still square-integrable; the growth envelope of the boundary driver $\psi$ increases by $L_g \|P\|_\infty$, which still belongs to $L^2(q)$ since $a_T < \infty$. Hence it suffices to prove the claim for $P = 0$.

\medskip
\noindent\textbf{Truncated equation and deterministic bound.}
Let $T_j(x) = (-j) \vee (x \wedge j)$,
\[
\eta_j = T_j(\eta), \qquad
F_j = F - F(\cdot, 0, 0) + T_j(F(\cdot, 0, 0)), \qquad
g_j = g - g(\cdot, 0) + T_j(g(\cdot, 0)).
\]
These preserve the original Lipschitz \textbf{(A1)}, continuity and dissipativity \textbf{(A2)} conditions. Let $\kappa^N = \kappa \wedge N$. For fixed $j, N$, we apply Theorem 2.1 of Feng \cite{feng2021generalized} with a driver independent of the copies, and deduce that the truncated equation
\begin{equation}
Y_t^{j,N} = \eta_j + \int_t^T F_j(s, Y_s^{j,N}, Z_s^{j,N}) \dif s
+ \int_t^T g_j(s, Y_s^{j,N}) \dif \kappa_s^N - \int_t^T Z_s^{j,N} \dif B_s
\label{eq:c5-aux-truncated}
\end{equation}
has a unique solution $(Y^{j,N}, Z^{j,N}) \in S^2 \times H^2$.

Fix $j$ and let $r_T = j \vee (j/b)$, $-r'_t = j + L r_t$. This deterministic function can be written as
\[
r_t = \begin{cases}
(r_T + j/L) e^{L(T-t)} - j/L, & L > 0, \\
r_T + j(T-t), & L = 0.
\end{cases}
\]
Hence $r_t \ge j/b$ and $r_t \le r_0$. On the set $\{Y^{j,N} > r\}$,
\[
F_j(t, Y^{j,N}, Z^{j,N}) + r'_t \le L(Y^{j,N} - r) + L|Z^{j,N}|,
\qquad
g_j(t, Y^{j,N}) \le -bY^{j,N} + j \le 0.
\]
Let $U = Y^{j,N} - r$. Then $U_T \le 0$ and
\[
\dif U_t = -(F_j(t, Y_t^{j,N}, Z_t^{j,N}) + r'_t) \dif t
- g_j(t, Y_t^{j,N}) \dif \kappa_t^N + Z_t^{j,N} \dif B_t.
\]
After localization, applying the It\^o--Tanaka formula to $(U^+)^2$ and using $U_T^+ = 0$, we have
\[
\begin{aligned}
(U_t^+)^2
={}& 2\int_t^T U_s^+ \bigl( F_j(s, Y_s^{j,N}, Z_s^{j,N}) + r'_s \bigr) \dif s
+ 2\int_t^T U_s^+ g_j(s, Y_s^{j,N}) \dif \kappa_s^N \\
& - 2\int_t^T U_s^+ Z_s^{j,N} \dif B_s
- \int_t^T \mathbf{1}_{\{U_s > 0\}} |Z_s^{j,N}|^2 \dif s.
\end{aligned}
\]
The boundary term is non-positive on $\{U > 0\}$, and the driver term satisfies
\[
2U^+\bigl(F_j(t, Y^{j,N}, Z^{j,N}) + r'_t\bigr) - \mathbf{1}_{\{U > 0\}}|Z|^2
\le 2L(U^+)^2 + \bigl(2LU^+|Z| - \mathbf{1}_{\{U > 0\}}|Z|^2\bigr)
\le (2L + L^2)(U^+)^2,
\]
where the last inequality follows from Young's inequality $2LU^+|Z| \le L^2(U^+)^2 + |Z|^2$.
Taking expectations and removing the localization yields
\[
\mathbb{E}(U_t^+)^2 \le (2L + L^2) \int_t^T \mathbb{E}(U_s^+)^2 \dif s.
\]
The backward Gronwall lemma gives $\mathbb{E}(U_t^+)^2 = 0$. Taking a common probability-one event at rational times and using continuity, we obtain $Y^{j,N} \le r$. Treating $-Y^{j,N}$ in the same way (the boundary driver $y \mapsto -g_j(t, -y)$ has the same dissipativity constant $b$), we obtain
\begin{equation}
\sup_N \sup_{t \le T} |Y_t^{j,N}| \le M_j := r_0 \quad \text{a.s.}
\label{eq:c5-aux-bound}
\end{equation}
By continuity and countably many $N$, a common probability-one event can be chosen.

\medskip
\noindent\textbf{Removing the truncation of $\kappa$.}
For $M > N$, set $\delta Y = Y^{j,M} - Y^{j,N}$ and $\delta Z = Z^{j,M} - Z^{j,N}$. Subtracting the equations, the common part $\kappa^N$ and the remaining part $\kappa^M - \kappa^N$ appear separately:
\[
\begin{aligned}
\delta Y_t = {} & \int_t^T \bigl(F_j(s, Y_s^{j,M}, Z_s^{j,M}) - F_j(s, Y_s^{j,N}, Z_s^{j,N})\bigr) \dif s \\
& + \int_t^T \bigl(g_j(s, Y_s^{j,M}) - g_j(s, Y_s^{j,N})\bigr) \dif \kappa_s^N \\
& + \int_t^T g_j(s, Y_s^{j,M}) \dif (\kappa_s^M - \kappa_s^N)
- \int_t^T \delta Z_s \dif B_s.
\end{aligned}
\]
On the common part $\kappa^N$, the boundary difference term is non-positive by dissipativity \textbf{(A2)}. The remaining part satisfies
\[
\dif(\kappa^M - \kappa^N)_s = \mathbf{1}_{\{N \le \kappa_s < M\}} \dif \kappa_s
\le \mathbf{1}_{\{\kappa_s \ge N\}} \dif \kappa_s.
\]
By the uniform bound \eqref{eq:c5-aux-bound}, $|\delta Y| \le 2M_j$ and $|g_j(s, Y_s^{j,M})| \le 2\psi_s + j + L_g M_j$. Consequently,
\[
\mathbb{E} \int_0^T |\delta Y_s| \, |g_j(s, Y_s^{j,M})| \dif (\kappa_s^M - \kappa_s^N)
\le C_j \mathbb{E} \int_0^T \mathbf{1}_{\{\kappa_s \ge N\}} (1 + \psi_s) \dif \kappa_s.
\]
Hence the weighted energy estimate gives
\[
\sup_t \mathbb{E}|\delta Y_t|^2 + \|\delta Z\|_{H^2}^2
+ b \mathbb{E} \int_0^T |\delta Y_s|^2 \dif \kappa_s^N \le C_j R_N^j,
\]
where $R_N^j = \mathbb{E} \int_0^T \mathbf{1}_{\{\kappa_s \ge N\}} (1 + \psi_s) \dif \kappa_s \to 0$; only $1 + \psi \in L^1(q)$ is used. The tail of the untruncated process also satisfies
\[
\mathbb{E} \int_0^T |\delta Y_s|^2 \dif (\kappa_s - \kappa_s^N)
\le 4M_j^2 \mathbb{E}(\kappa_T - N)^+ \le 4M_j^2 R_N^j.
\]
Taking the pathwise supremum in the difference It\^o formula, the dissipative term on the common part can still be dropped, and the tail is controlled by $C_j R_N^j$. Let $S_{M,N} = \mathbb{E} \sup_t |\delta Y_t|^2$. The driver difference is controlled by the energy estimate, and the martingale term satisfies
\[
\mathbb{E} \sup_t \left| \int_t^T e^{\alpha s} \delta Y_s \delta Z_s \dif B_s \right|
\le \tfrac14 S_{M,N} + C \mathbb{E} \int_0^T |\delta Z_s|^2 \dif s.
\]
Hence $S_{M,N} \le C_j R_N^j + \tfrac14 S_{M,N}$. Absorbing and adding back the tail of the untruncated process gives
\[
\|\delta Y\|_{S^2}^2 + \|\delta Z\|_{H^2}^2
+ b \|\delta Y\|_{L^2(q)}^2 \le C_j R_N^j.
\]
By Lemma \ref{lem:limit-identification}, we obtain a limit $Y^j$ in both $S^2$ and $L^2(q)$ for the same process. The driver is passed to the limit using its Lipschitz property \textbf{(A1)}, and the boundary driver using Lemma \ref{lem:nemytskii}. Moreover,
\[
\mathbb{E} \sup_t \left| \int_0^t H_s \dif \kappa_s \right|
\le \sqrt{a_T} \|H\|_{L^2(q)};
\]
concretely, the difference of the boundary cumulative integrals decomposes as
\[
\begin{aligned}
& \int_0^t g_j(s, Y_s^{j,N}) \dif \kappa_s^N - \int_0^t g_j(s, Y_s^j) \dif \kappa_s \\
= {} & \int_0^t \bigl(g_j(s, Y_s^{j,N}) - g_j(s, Y_s^j)\bigr) \dif \kappa_s^N
- \int_0^t g_j(s, Y_s^j) \dif (\kappa_s - \kappa_s^N).
\end{aligned}
\]
The $L^1$ pathwise norm of the first term is controlled by $\sqrt{a_T} \|g_j(\cdot, Y^{j,N}) - g_j(\cdot, Y^j)\|_{L^2(q)}$, while the second term tends to zero by the dominated convergence theorem, since $g_j(\cdot, Y^j) \in L^1(q)$. This argument uses only the first moment of $\kappa$. The stochastic integral is handled by the BDG inequality, and we obtain
\[
Y_t^j = \eta_j + \int_t^T F_j(s, Y_s^j, Z_s^j) \dif s
+ \int_t^T g_j(s, Y_s^j) \dif \kappa_s - \int_t^T Z_s^j \dif B_s.
\]

\medskip
\noindent\textbf{Removing the data truncation.}
For the data difference, define
\[
\begin{aligned}
R_{j,\ell} = {} & \mathbb{E}|\eta_j - \eta_\ell|^2
+ \|T_j(F(\cdot, 0, 0)) - T_\ell(F(\cdot, 0, 0))\|_{L^2(\mathbb{P} \otimes \dif t)}^2 \\
& + b^{-1} \|T_j(g(\cdot, 0)) - T_\ell(g(\cdot, 0))\|_{L^2(q)}^2 \to 0.
\end{aligned}
\]
The dissipative difference calculation with the same $\kappa$ gives
\[
\sup_t \mathbb{E}|Y_t^j - Y_t^\ell|^2 + \|Z^j - Z^\ell\|_{H^2}^2
+ b \|Y^j - Y^\ell\|_{L^2(q)}^2 \le C R_{j,\ell}.
\]
In the pathwise estimate we keep
\[
2\delta Y_s \bigl(g_j(s, Y_s^j) - g_\ell(s, Y_s^\ell)\bigr)
\le -b|\delta Y_s|^2 + b^{-1}|T_j(g(s, 0)) - T_\ell(g(s, 0))|^2.
\]
Combining the BDG inequality and the energy estimate, we obtain
\[
\|Y^j - Y^\ell\|_{S^2}^2 + \|Z^j - Z^\ell\|_{H^2}^2
+ b \|Y^j - Y^\ell\|_{L^2(q)}^2 \le C R_{j,\ell}.
\]
By the dominated convergence theorem, $R_{j,\ell} \to 0$, so the two $Y$ norms and the $Z$ norm have strong limits. Lemma \ref{lem:limit-identification} identifies the two limits of $Y$. The convergence of the driver splits into the state difference and the inhomogeneous data difference:
\[
\begin{aligned}
& F_j(\cdot, Y^j, Z^j) - F(\cdot, Y, Z) \\
= {} & F(\cdot, Y^j, Z^j) - F(\cdot, Y, Z)
+ T_j(F(\cdot, 0, 0)) - F(\cdot, 0, 0).
\end{aligned}
\]
Both terms tend to zero in $L^2(\mathbb{P} \otimes \dif t)$. The boundary term is written similarly as
\[
g_j(\cdot, Y^j) - g(\cdot, Y)
= g(\cdot, Y^j) - g(\cdot, Y) + T_j(g(\cdot, 0)) - g(\cdot, 0).
\]
The first term tends to zero in $L^2(q)$ by Nemytskii continuity, and the second by the convergence of the data truncations. Using the cumulative integral bound above, we pass to the $L^1$ pathwise limit in the dynamics and obtain the full-process version of the original equation. When the data difference is zero, the same energy estimate gives uniqueness.

Finally, we undo the deterministic shift. Since $P$ is continuous and bounded,
\[
\mathbb{E} \int_0^T |P_t|^2 \dif \kappa_t \le \|P\|_\infty^2 a_T < \infty,
\]
so $Y = \widetilde Y - P$ retains the two required square-integrability properties. This completes the proof.
\end{proof}

\section{Construction of the solution by penalization}
\label{sec:construction}

In this section, we first construct a smooth obstacle that preserves terminal compatibility, then freeze the distribution and the mean feedback, and apply the auxiliary well-posedness lemma of the previous section to prove that the penalized equation with fixed parameters is solvable.

\subsection{Smooth approximation of the obstacle preserving terminal compatibility}
\label{sec:mollifier}

The mollification of a continuous obstacle requires a downward shift in order to preserve terminal compatibility. Extend $u$ constantly outside $[0,T]$. Take a non-negative smooth function $\rho$ supported in $[-1,1]$ with $\int\rho=1$. For $k\ge1$, set
\[
\rho_k(t)=k\rho(kt),\qquad\widetilde u_t^k=(u*\rho_k)(t),\qquad
\varepsilon_k=\|\widetilde u^k-u\|_\infty,\qquad u_t^k=\widetilde u_t^k-\varepsilon_k.
\]

The following lemma summarizes the basic properties of the approximation sequence $(u^k)_{k\ge1}$. These properties guarantee terminal compatibility.

\begin{lemma}[Properties of the smooth approximation]
\label{lem:mollifier}
The approximation sequence $(u^k)_{k\ge1}$ satisfies:
\begin{enumerate}[label=(\roman*)]
\item each $u^k \in C^\infty([0,T])$, and all derivatives are bounded on $[0,T]$;
\item $u_t^k \le u_t$ for all $t$, and $\|u^k-u\|_\infty \le 2\varepsilon_k \to 0$;
\item $u_T^k \le u_T \le \mathbb{E}\xi$;
\item there exists a constant $C_k$, depending only on $k$, such that $\|(u^k)'\|_{L^2(\dif t)}+\sup_t|u_t^k|\le C_k$.
\end{enumerate}
\end{lemma}

\begin{proof}
The constant extension is uniformly continuous on the real line, so the convolution is smooth and converges uniformly. Since $\|\widetilde u^k-u\|_\infty=\varepsilon_k$, we have $u^k\le u$ and $\|u^k-u\|_\infty\le2\varepsilon_k\to0$, which also guarantees terminal compatibility. All derivatives are bounded on compact intervals, so the squared integral of the derivative is finite.
\end{proof}

\subsection{The penalized equations and their well-posedness}
\label{sec:penalized}

For a fixed smooth obstacle and $n\ge1$, consider
\begin{equation}
\begin{aligned}
Y_t^{n,k}={}&\xi+\int_t^T f(s,Y_s^{n,k},Z_s^{n,k},\nu_s^{n,k})\dif s
+\int_t^T g(s,Y_s^{n,k})\dif\kappa_s\\
&+K_T^{n,k}-K_t^{n,k}-\int_t^T Z_s^{n,k}\dif B_s,
\end{aligned}
\label{eq:penalized}
\end{equation}
where $\nu_s^{n,k}=\mathcal L(Y_s^{n,k},Z_s^{n,k})$, and the deterministic penalty process is
\begin{equation}
K_t^{n,k}=n\int_0^t(\mathbb{E}Y_s^{n,k}-u_s^k)^-\dif A_s.
\label{eq:c5-penalty-A}
\end{equation}
Set $v_s^{n,k}=(\mathbb{E}Y_s^{n,k}-u_s^k)^-$. Then
\begin{equation}
\int_0^T(\mathbb{E}Y_s^{n,k}-u_s^k)\dif K_s^{n,k}
=-n\int_0^T|v_s^{n,k}|^2\dif A_s\le0.
\label{eq:flatness_nk}
\end{equation}

To construct the solution of the original equation, we first need to show that each fixed penalized equation is well posed. The following proposition proves this by using the well-posedness of the auxiliary generalized BSDE, freezing the distribution and the mean feedback, and applying the Banach fixed point theorem.

\begin{proposition}[Well-posedness of the penalized equation]
\label{prop:c5-penal}
Under assumption \textbf{(A)}, for any $n,k\ge1$, the equations \eqref{eq:penalized}--\eqref{eq:c5-penalty-A} have a unique solution $(Y^{n,k},Z^{n,k})\in(S^2\cap L^2(q))\times H^2$, with the corresponding $K^{n,k}\in A_d^2$.
\end{proposition}

\begin{proof}
On the complete space $\mathcal H_A=L^2_{\mathrm{prog}}(\mathbb{P}\otimes\dif A)\times H^2$, given $(U,V)$, freeze $\nu^{U,V}=\mathcal L(U,V)$ and $p_s^U=n(\mathbb{E}U_s-u_s^k)^-$. We have $p^U\in L^2(\dif A)$, and the frozen driver is square-integrable at the origin because $W_2^2(\nu^{U,V},\delta_{(0,0)})\le\mathbb{E}|U|^2+\mathbb{E}|V|^2$. Lemma \ref{lem:c5-aux} gives the unique solution of the frozen equation
\[
\begin{aligned}
Y_t={}&\xi+\int_t^T f(s,Y_s,Z_s,\nu_s^{U,V})\dif s
+\int_t^T g(s,Y_s)\dif\kappa_s\\
&+\int_t^T p_s^U\dif A_s-\int_t^T Z_s\dif B_s,
\end{aligned}
\]
which defines a map $\Phi$. The output $Y$ of this map belongs to $S^2$, hence also to $L^2(\mathbb{P}\otimes\dif A)$, so the map takes this space into itself.

Take $w_s=e^{\lambda A_s}$, and set $N_A(H)=\mathbb{E}\int w|H|^2\dif A$ and $N_t(H)=\mathbb{E}\int w|H|^2\dif t$. The weighted difference It\^o formula, boundary dissipativity \textbf{(A2)} and $|p^U-p^{\bar U}|\le n\|\delta U\|_{L^2}$ give
\[
\tfrac12N_t(\delta Z)+(\lambda/2-c_0)N_A(\delta Y)
\le\frac{4(L_f^2+n^2)}{\lambda}(N_A(\delta U)+N_t(\delta V)),\quad c_0=2L_f+2L_f^2.
\]
When $\lambda\ge2c_0+1$ and $\lambda>8(L_f^2+n^2)$,
\[
N_A(\delta Y)+N_t(\delta Z)
\le\frac{8(L_f^2+n^2)}{\lambda}(N_A(\delta U)+N_t(\delta V)).
\]
Since $A_T<\infty$ is deterministic, the weighted norms are equivalent to the unweighted ones. The Banach fixed point theorem gives a unique fixed point.

Let $(Y,Z)$ be this fixed point, and take a continuous version of $Y$. The fixed point relation holds in $\mathbb{P}\otimes\dif A$ and $\mathbb{P}\otimes\dif t$, respectively, so the mean feedback and the distribution term are consistent with the penalized equation. The process $K_t=\int_0^t p_s^U\dif A_s$ is deterministic, continuous, non-decreasing and finite, hence belongs to $A_d^2$. Conversely, any solution of the penalized equation is a fixed point of the map, whence uniqueness.
\end{proof}

\begin{remark}
\label{rem:c5-penal-params}
The contraction parameter $\lambda$ depends on $n$; the argument above gives well-posedness for fixed parameters. Uniform a priori bounds with respect to $n,k$ have to be established separately.
\end{remark}

\subsection{Uniform a priori estimates for the penalized solutions}
\label{sec:a_priori_penal}

The fixed-parameter contraction argument allows constants to depend on $n$. Below, we shift the obstacle by a constant upper bound to obtain estimates that are independent of $n,k$ and of the derivative of the obstacle.

\begin{lemma}[Uniform a priori bounds for the penalized equations]
\label{lem:a_priori_penal}
Under assumption \textbf{(A)}, for any $n,k\ge1$, let $(Y^{n,k},Z^{n,k})$ be the solution of \eqref{eq:penalized}. Set
\[
\begin{aligned}
M&:=\sup_{t\le T}u_t^+,\\
\Theta&:=\|(\xi, f^0)\|^2
+\mathbb{E}\int_0^T\psi_t^2\dif\kappa_t+M^2(1+T+a_T),
\end{aligned}
\]
where $f_t^0=f(t,0,0,\delta_{(0,0)})$. There exists a constant $C>0$, depending only on $L_f,L_g,b,T,a_T$, such that
\begin{equation}
\sup_{n,k}\left\{\sup_t\mathbb{E}|Y_t^{n,k}|^2
+\|Z^{n,k}\|_{H^2}^2+\|Y^{n,k}\|_{L^2(q)}^2+(K_T^{n,k})^2\right\}\le C\Theta,
\label{eq:uniform-energy}
\end{equation}
and
\begin{equation}
\sup_{n,k}\left\{\|Y^{n,k}\|_{S^2}^2
+\|Z^{n,k}\|_{H^2}^2+\|Y^{n,k}\|_{L^2(q)}^2+(K_T^{n,k})^2\right\}\le C\Theta.
\label{eq:full-uniform}
\end{equation}
\end{lemma}

\begin{proof}
Let $X=Y^{n,k}-M$. Since $u^k\le u\le M$,
\[
\mathbb{E}X_s\dif K_s^{n,k}\le-n|v_s^{n,k}|^2\dif A_s\le0.
\]
Monotonicity \textbf{(A2)} of the boundary driver gives
\[
2X_sg(s,X_s+M)\le-b|X_s|^2+\tfrac2b(\psi_s^2+L_g^2M^2).
\]
The driver is handled as in the a priori estimate, and the inhomogeneous term increases only by $CM^2$. Applying the weighted It\^o formula, taking expectations and absorbing, we obtain
\begin{equation}
\sup_t\mathbb{E}|Y_t^{n,k}|^2+\|Z^{n,k}\|_{H^2}^2+\|Y^{n,k}\|_{L^2(q)}^2\le C\Theta.
\label{eq:uniform-energy-step1}
\end{equation}
Set $f^{n,k}=f(\cdot,Y^{n,k},Z^{n,k},\nu^{n,k})$ and $g^{n,k}=g(\cdot,Y^{n,k})$. The expectation representation of the reflection gives
\[
K_T^{n,k}\le|\mathbb{E}Y_0^{n,k}|+|\mathbb{E}\xi|
+\sqrt T\|f^{n,k}\|_{L^2(\mathbb{P}\otimes\dif t)}
+\sqrt{a_T}\|g^{n,k}\|_{L^2(q)}\le C\sqrt\Theta.
\]
The last bound follows from the Lipschitz condition \eqref{eq:c5-A1-lipschitz}, the growth condition \eqref{eq:c5-A2-growth} and estimate \eqref{eq:uniform-energy-step1}. Hence
\[
\sup_{n,k}(K_T^{n,k})^2\le C\Theta,
\]
which yields estimate \eqref{eq:uniform-energy}. Substituting the bound on $K_T$ into Proposition \ref{prop:c5-priori}, which uses only the dynamics, gives estimate \eqref{eq:full-uniform}.
\end{proof}

\subsection{Error estimate}
\label{sec:error_estimate_penal}

Below, using the density of the expected integral with respect to $\kappa$, we reduce the evolution of the mean deviation to a finite-variation equation with respect to the deterministic function $A$, and estimate its negative part.

\begin{lemma}[Error estimate]
\label{lem:error_estimate_penal}
Under assumption \textbf{(A)}, let $k\ge1$ and $n\ge1$ be given, and let $(Y^{n,k},Z^{n,k})$ be the solution of the penalized equation. Set
\[
y_t:=\mathbb{E}[Y_t^{n,k}]-u_t^k,\qquad v_t:=y_t^-.
\]
Then there exists a constant $C_k>0$, independent of $n$, such that
\begin{equation}
\sup_{t\le T}v_t^2\le C_k/n,\qquad
\int_0^T v_t^2\dif A_t\le C_k/n^2.
\label{eq:negative-error}
\end{equation}
\end{lemma}

\begin{proof}
Set $F_t=\mathbb{E}f(t,Y_t^{n,k},Z_t^{n,k},\nu_t^{n,k})$ and $g_t^{n,k}=g(t,Y_t^{n,k})$. By Lemma \ref{lem:c5-density}, the expected integral of $g^{n,k}$ with respect to $\kappa$ has a density $h$, and
\[
\mathbb{E}\int_0^t g_s^{n,k}\dif\kappa_s=\int_0^t h_s\dif a_s,
\qquad \|h\|_{L^2(\dif a)}^2\le\|g^{n,k}\|_{L^2(q)}^2.
\]
Taking expectations in the penalized equation and subtracting $u^k$, we obtain
\[
y_t=y_T+\int_t^T(F_s+(u^k)'_s)\dif s
+\int_t^T h_s\dif a_s+n\int_t^T v_s\dif A_s.
\]
With the functions $\theta,\eta$ defined in Section \ref{sec:kappa}, set
\[
Q_t=\theta_t(F_t+(u^k)'_t)+\eta_t h_t.
\]
Terminal compatibility gives $y_T\ge0$, so that
\[
\dif y_t=-(Q_t+nv_t)\dif A_t,\qquad y_T\ge0.
\]
Applying the chain rule to $(y^-)^2$, using $v_T=0$ and Young's inequality, we obtain
\begin{equation}
v_t^2+n\int_t^T v_s^2\dif A_s\le n^{-1}\int_0^T|Q_s|^2\dif A_s.
\label{eq:c5-error-key}
\end{equation}
By Jensen's inequality and the Lipschitz property \textbf{(A1)} of the driver,
\[
|F_t|^2\le C\bigl(\mathbb{E}|f_t^0|^2+
\mathbb{E}|Y_t^{n,k}|^2+\mathbb{E}|Z_t^{n,k}|^2\bigr).
\]
Combining the convexity estimate with $\theta+\eta=1$, the growth condition \textbf{(A2)} and the uniform a priori bounds,
\[
\|Q\|_{L^2(\dif A)}^2\le2\|F\|_{L^2(\dif t)}^2
+2\|(u^k)'\|_{L^2(\dif t)}^2+\|g^{n,k}\|_{L^2(q)}^2\le C_k.
\]
Substituting this into inequality \eqref{eq:c5-error-key} gives the conclusion. Here $C_k$ may depend on the derivative of the obstacle, and $h$ is the density of the expected integral with respect to $\kappa$.
\end{proof}

\subsection{Cauchy property of the penalized sequence}
\label{sec:cauchy_penal}

Once the smooth obstacle is fixed, the reflection cross term is controlled by the negative-part error. The difference estimate retains the $L^2(q)$ dissipative term, so that one can pass to the limit in the boundary driver.

\begin{lemma}[Cauchy property of the penalized sequence]
\label{lem_cauchy_penal}
Under assumption \textbf{(A)}, set $\delta Y_t:=Y_t^{m,k}-Y_t^{n,k}$, $\delta Z_t:=Z_t^{m,k}-Z_t^{n,k}$ and $r_{m,n}:=m^{-1}+n^{-1}$. For fixed $k\ge1$, there exists a constant $C_k>0$, independent of $m,n$, such that for any $m,n\ge1$,
\begin{equation}
\sup_t\mathbb{E}|\delta Y_t|^2+\mathbb{E}\int_0^T|\delta Z_s|^2\dif s
+b\mathbb{E}\int_0^T|\delta Y_s|^2\dif\kappa_s\le C_kr_{m,n},
\label{eq:penal-energy-diff}
\end{equation}
and
\begin{equation}
\mathbb{E}\sup_t|\delta Y_t|^2+\mathbb{E}\int_0^T|\delta Z_s|^2\dif s
\le C_k(m^{-1/2}+n^{-1/2}).
\label{eq:penal-S2-diff}
\end{equation}
Consequently, the sequence $(Y^{m,k},Z^{m,k})_{m\ge1}$ is Cauchy in $(S^2\cap L^2(q))\times H^2$.
\end{lemma}

\begin{proof}
Set $\delta K=K^{m,k}-K^{n,k}$ and $\delta f=f(\cdot,Y^{m,k},Z^{m,k},\nu^{m,k})-f(\cdot,Y^{n,k},Z^{n,k},\nu^{n,k})$. The weighted difference It\^o formula reads
\[
\begin{aligned}
&\mathbb{E}e^{\alpha t}|\delta Y_t|^2+
\mathbb{E}\int_t^T e^{\alpha s}(|\delta Z_s|^2+\alpha|\delta Y_s|^2)\dif s\\
&=2\mathbb{E}\int_t^T e^{\alpha s}\delta Y_s\delta f_s\dif s\\
&\quad+2\mathbb{E}\int_t^T e^{\alpha s}\delta Y_s
\bigl(g(s,Y_s^{m,k})-g(s,Y_s^{n,k})\bigr)\dif\kappa_s\\
&\quad+2\mathbb{E}\int_t^T e^{\alpha s}\delta Y_s\dif\delta K_s.
\end{aligned}
\]
Boundary dissipativity \textbf{(A2)} and the driver estimate give, respectively,
\[
2\mathbb{E}\int_t^T e^{\alpha s}\delta Y_s
\bigl(g(s,Y_s^{m,k})-g(s,Y_s^{n,k})\bigr)\dif\kappa_s
\le-2b\mathbb{E}\int_t^T e^{\alpha s}|\delta Y_s|^2\dif\kappa_s,
\]
\[
2\delta Y_s\delta f_s\le C_f|\delta Y_s|^2+\tfrac14|\delta Z_s|^2
+\tfrac14(\mathbb{E}|\delta Y_s|^2+\mathbb{E}|\delta Z_s|^2).
\]
Set $v^\ell=(\mathbb{E}Y^{\ell,k}-u^k)^-$. Expanding the reflection pairing and using the weighted equality \eqref{eq:flatness_nk} to discard the two non-positive self-pairing terms, we obtain
\[
2\mathbb{E}\int_t^T e^{\alpha s}\delta Y_s\dif\delta K_s
\le2(m+n)e^{\alpha T}\int_0^T v_s^mv_s^n\dif A_s.
\]
By the error estimate \eqref{eq:negative-error} and the Cauchy--Schwarz inequality,
\[
2\mathbb{E}\int_t^T e^{\alpha s}\delta Y_s\dif\delta K_s\le C_kr_{m,n}.
\]
In the general difference estimate \eqref{eq:c5-difference-budget}, with data difference $R=0$ and $E_K=C_kr_{m,n}$, we obtain estimate \eqref{eq:penal-energy-diff}. Moreover, the uniform a priori bounds give $\Lambda=K_T^{m,k}+K_T^{n,k}\le C$, so the pathwise reflection term satisfies
\[
2\mathbb{E}\sup_t\left|\int_t^T e^{\alpha s}\delta Y_s\dif\delta K_s\right|
\le2e^{\alpha T}\Lambda\sup_s(\mathbb{E}|\delta Y_s|^2)^{1/2}
\le C_k\sqrt{r_{m,n}}.
\]
The same difference estimate gives $\|\delta Y\|_{S^2}^2+\|\delta Z\|_{H^2}^2\le C_k(r_{m,n}+\sqrt{r_{m,n}})\le C_k(m^{-1/2}+n^{-1/2})$, which is estimate \eqref{eq:penal-S2-diff}. Lemma \ref{lem:limit-identification} identifies the limits in the two $Y$ norms as the same process.
\end{proof}

\begin{remark}
\label{rem:c5-cauchy-rates}
For fixed $k$, the differences of $Z$ in $H^2$ and of $Y$ in $L^2(q)$ are of order $O_k(m^{-1/2}+n^{-1/2})$, while the difference of $Y$ in $S^2$ is of order $O_k(m^{-1/4}+n^{-1/4})$. These rates follow from the upper bounds above, and their optimality is not discussed here.
\end{remark}

\section{Existence proof}
\label{sec:exist_solution}

In this section, we first let $n\to\infty$ with $k$ fixed to obtain the solution associated with the smooth obstacle $u^k$, and then let $k\to\infty$ to obtain the solution of the original equation. The two limiting procedures share the following argument for the transfer of the dynamics and the Skorokhod condition.

\begin{lemma}[Transfer of the dynamics and the Skorokhod condition in the limit]
\label{lem:c5-limit-transfer}
Let $(Y^j,Z^j,K^j)$ satisfy the dynamics with the same data, with $K^j$ deterministic, continuous, non-decreasing, $K_0^j=0$, and $\sup_jK_T^j<\infty$. If
\[
Y^j\to Y\quad\text{in }S^2\cap L^2(q),\qquad
Z^j\to Z\quad\text{in }H^2,
\]
then $K^j$ converges uniformly to the process $K\in A_d^2$ defined by the expectation representation, and the limit satisfies the same dynamics. If, in addition, continuous deterministic obstacles $u^j\to u$ converge uniformly and
\[
\|(\mathbb{E}Y^j-u^j)^-\|_\infty\to0,\qquad
\int_0^T(\mathbb{E}Y_t^j-u_t^j)\dif K_t^j\to0,
\]
then the limit satisfies the mean constraint and the Skorokhod condition.
\end{lemma}

\begin{proof}
Set $f^j=f(\cdot,Y^j,Z^j,\nu^j)$ and $g^j=g(\cdot,Y^j)$. The Lipschitz property \textbf{(A1)} gives $f^j\to f$ in $L^2(\mathbb{P}\otimes\dif t)$, and Lemma \ref{lem:nemytskii} gives $g^j\to g$ in $L^2(q)$. Define
\[
K_t=\mathbb{E}Y_0-\mathbb{E}Y_t-\int_0^t\mathbb{E}f_s\dif s
-\mathbb{E}\int_0^t g_s\dif\kappa_s.
\]
The expectation representation gives
\[
\|K^j-K\|_\infty\le2\sup_t\mathbb{E}|Y_t^j-Y_t|
+\sqrt T\|f^j-f\|_{L^2(\mathbb{P}\otimes\dif t)}
+\sqrt{a_T}\|g^j-g\|_{L^2(q)}\to0.
\]
Hence $K$ is continuous, non-decreasing, starts at zero and has finite terminal value. The terminal value is inherited from the $S^2$ convergence. The dynamic integrals converge in $L^1(\Omega;C([0,T]))$, where the boundary integral uses
\[
\mathbb{E}\sup_t\left|\int_0^t(g_s^j-g_s)\dif\kappa_s\right|
\le\sqrt{a_T}\|g^j-g\|_{L^2(q)},
\]
and the stochastic integral is handled by the BDG inequality. Therefore the limiting dynamics holds simultaneously for all times.

The function $\mathbb{E}Y^j-u^j$ converges uniformly to the continuous function $\mathbb{E}Y-u$, and the negative-part condition gives its non-negativity. Since $K^j\to K$ uniformly with uniformly bounded terminal values, the Helly--Bray theorem yields the convergence of the Stieltjes integrals, so that the limit satisfies the Skorokhod condition.
\end{proof}

\subsection{First limit: solution for the smooth obstacle}
\label{sec:first-limit}

We first show the continuity of the means, which is needed to apply the transfer lemma above.

\begin{lemma}[Continuity of the mean function]
\label{lem:c5-mean-cont}
Under assumption \textbf{(A)}:
\begin{enumerate}[label=(\roman*)]
  \item for every pair $(n,k)$, the function $t\mapsto\mathbb{E}[Y_t^{n,k}]$ is continuous on $[0,T]$;
  \item if $Y\in S^2$, then $t\mapsto\mathbb{E}[Y_t]$ is continuous on $[0,T]$.
\end{enumerate}
\end{lemma}

\begin{proof}
If $t_j\to t$, path continuity gives $Y_{t_j}\to Y_t$ almost surely, and $|Y_{t_j}-Y_t|\le2\sup_s|Y_s|\in L^1$. The dominated convergence theorem gives (ii). Since the penalized solutions belong to $S^2$, applying (ii) gives (i).
\end{proof}

After these preparations, we fix $k$ and pass to the limit in the penalization parameter $n$, which yields the solution associated with the smooth obstacle $u^k$. The following proposition gives the existence of this solution and its relation to the penalized solutions.

\begin{proposition}[Existence for the smooth obstacle]
\label{prop:exist-uk}
Under assumption \textbf{(A)}, for every given $k\ge1$, the mean-reflected MF-GBSDE associated with the approximating obstacle $u^k$ has a solution $(Y^k,Z^k,K^k)$ satisfying the dynamics
\[
Y_t^k=\xi+\int_t^T f(s,Y_s^k,Z_s^k,\nu_s^k)\dif s
+\int_t^T g(s,Y_s^k)\dif\kappa_s+K_T^k-K_t^k-\int_t^T Z_s^k\dif B_s,
\]
the reflection condition $\mathbb{E}[Y_t^k]\ge u_t^k$, and the Skorokhod condition
\[
\int_0^T(\mathbb{E}[Y_t^k]-u_t^k)\dif K_t^k=0,
\]
where $K^k\in A_d^2$. Moreover, as $n\to\infty$,
\[
Y^{n,k}\to Y^k\quad\text{in }S^2\cap L^2(q),
\qquad Z^{n,k}\to Z^k\quad\text{in }H^2,
\]
and $K^{n,k}\to K^k$ uniformly, with $\sup_n K_T^{n,k}<\infty$.
\end{proposition}

\begin{proof}
Fix $k$. By Lemma \ref{lem_cauchy_penal} and Lemma \ref{lem:limit-identification}, we obtain the stated strong convergence of $(Y^k,Z^k)$. Set $f^{n,k}=f(\cdot,Y^{n,k},Z^{n,k},\nu^{n,k})$ and $g^{n,k}=g(\cdot,Y^{n,k})$, and define $f^k,g^k$ similarly. The Lipschitz property \textbf{(A1)} and the Nemytskii continuity (Lemma \ref{lem:nemytskii}) give, respectively,
\[
f^{n,k}\to f^k\quad\text{in }L^2(\mathbb{P}\otimes\dif t),
\]
\[
g^{n,k}\to g^k\quad\text{in }L^2(q).
\]
The uniform a priori bounds control $K_T^{n,k}$. On the other hand, the error estimate \eqref{eq:negative-error} gives $\|(\mathbb{E}Y^{n,k}-u^k)^-\|_\infty\to0$, and by equality \eqref{eq:flatness_nk},
\[
\left|\int_0^T(\mathbb{E}Y_t^{n,k}-u_t^k)\dif K_t^{n,k}\right|
=n\int_0^T|v_t^{n,k}|^2\dif A_t\le C_k/n\to0.
\]
Applying Lemma \ref{lem:c5-limit-transfer}, we obtain $K^k$, the dynamics, the mean constraint and the Skorokhod condition. Here
\[
K_t^k=\mathbb{E}Y_0^k-\mathbb{E}Y_t^k-\int_0^t\mathbb{E}f_s^k\dif s
-\mathbb{E}\int_0^t g_s^k\dif\kappa_s,
\]
and the specialized uniform convergence estimate holds:
\[
\begin{aligned}
\|K^{n,k}-K^k\|_\infty\le{}&2\|Y^{n,k}-Y^k\|_{S^2}
+\sqrt T\|f^{n,k}-f^k\|_{L^2(\mathbb{P}\otimes\dif t)}\\
&+\sqrt{a_T}\|g^{n,k}-g^k\|_{L^2(q)}\to0.
\end{aligned}
\]
\end{proof}

\subsection{Second limit: solution of the original equation}
\label{sec:second-limit}

We now use the uniform approximation of the obstacle to prove the Cauchy property of the sequence of solutions for the smooth obstacles, and thereby obtain the solution of the original equation.

\begin{proposition}[Existence for the original equation]
\label{prop:exist-u}
Under assumption \textbf{(A)}, equation \eqref{eq:mf-gbsde} has a solution $(Y,Z,K)\in S^2\times H^2\times A_d^2$.
\end{proposition}

\begin{proof}
The first limit and the uniform penalization bounds give
\[
\sup_k\left(\|Y^k\|_{S^2}^2+\|Z^k\|_{H^2}^2
+\|Y^k\|_{L^2(q)}^2+(K_T^k)^2\right)\le C\Theta.
\]
Set $\delta Y=Y^k-Y^m$, $\delta Z=Z^k-Z^m$, $\delta K=K^k-K^m$, $h^i=\mathbb{E}Y^i-u^i$, $\varepsilon_{k,m}=\|u^k-u^m\|_\infty$ and $\Lambda_{k,m}=K_T^k+K_T^m\le C_\Theta$. By the Skorokhod condition,
\[
\begin{aligned}
2\int_t^T e^{\alpha s}\mathbb{E}\delta Y_s\dif\delta K_s
={}&-2\int_t^T e^{\alpha s}h_s^k\dif K_s^m
-2\int_t^T e^{\alpha s}h_s^m\dif K_s^k\\
&+2\int_t^T e^{\alpha s}(u_s^k-u_s^m)\dif\delta K_s.
\end{aligned}
\]
The first two terms are non-positive by feasibility, so
\[
2\mathbb{E}\int_t^T e^{\alpha s}\delta Y_s\dif\delta K_s
\le2e^{\alpha T}\varepsilon_{k,m}\Lambda_{k,m}\le C_\Theta\varepsilon_{k,m}.
\]
In the general difference estimate \eqref{eq:c5-difference-budget}, taking $R=0$ and $E_K=C_\Theta\varepsilon_{k,m}$, we obtain
\[
\sup_t\mathbb{E}|\delta Y_t|^2+\|\delta Z\|_{H^2}^2
+b\|\delta Y\|_{L^2(q)}^2\le C_\Theta\varepsilon_{k,m},
\]
and
\[
\|\delta Y\|_{S^2}^2+\|\delta Z\|_{H^2}^2
\le C_\Theta(\varepsilon_{k,m}+\sqrt{\varepsilon_{k,m}}).
\]
Hence, by the uniform approximation of the obstacle and the limit identification lemma, $(Y^k,Z^k)$ converges in the required intersection space to a common $(Y,Z)$. Since $u^k\to u$ uniformly, the negative-part errors of the solutions are zero and the Skorokhod condition holds for each of them. Applying Lemma \ref{lem:c5-limit-transfer}, we obtain $K\in A_d^2$, the original dynamics and all the reflection conditions.
\end{proof}

\subsection{Proof of the well-posedness theorem}
\label{sec:proof-main}

\begin{proof}[Proof of Theorem \ref{thm:existence_uniqueness}]
Existence and uniqueness follow from Proposition \ref{prop:exist-u} and Proposition \ref{prop:uniqueness}, respectively. The two limiting procedures preserve the uniform a priori bounds above; the corresponding norms converge strongly and $K$ converges uniformly, so estimate \eqref{eq:c5-main-estimate} holds.
\end{proof}

\section{A Markovian PDE representation}
\label{sec:pde}

The general equation \eqref{eq:mf-gbsde} allows the drivers to depend on the full randomness, and the process $\kappa$ need not be a functional of a finite-dimensional Markov state. A deterministic finite-dimensional PDE representation therefore requires additional structure: a forward Markov process, deterministic state-dependent coefficients, a Markovian terminal value, a Markovian interpretation of $\kappa$, and a fixed reference flow $\mu_t$ computed once from the forward process. In this section we carry out this program in a concrete class. The forward process is reflected Brownian motion on the interval $[0,1]$, the process $\kappa$ is its boundary local time, the boundary driver is linear and dissipative, and the interior driver is a distribution-dependent Lipschitz function of the state. In this class the mean constraint becomes an integral constraint on the solution of a nonlocal parabolic problem, and the reflection process becomes a deterministic time-measure source. We prove well-posedness and a probabilistic representation for the resulting problem, and we describe the regularity of the multiplier: it is absolutely continuous when the obstacle is absolutely continuous, while a merely continuous obstacle can produce a singular multiplier.

\subsection{The Markovian class}
\label{sec:pde-class}

Fix $T>0$, a point $x_*\in[0,1]$ and a constant $\gamma>0$. We take $d=1$ and define the forward process by the Skorokhod reflection of the Brownian motion $B$ on $[0,1]$:
\begin{equation}
X_t=x_*+B_t+L^0_t-L^1_t,
\label{eq:c5-pde-forward}
\end{equation}
where $L^0$ and $L^1$ are continuous non-decreasing processes with $L^0_0=L^1_0=0$,
\[
\int_0^T\mathbf{1}_{\{X_t>0\}}\dif L^0_t=0,\qquad
\int_0^T\mathbf{1}_{\{X_t<1\}}\dif L^1_t=0.
\]
The classical Skorokhod construction gives a unique strong solution of \eqref{eq:c5-pde-forward}, adapted to the Brownian filtration. Set
\[
\kappa_t:=L^0_t+L^1_t,\qquad\mu_t:=\mathcal{L}(X_t).
\]
For $t>0$, the law $\mu_t$ has the density
\begin{equation}
p(t,x)=1+2\sum_{n=1}^{\infty}e^{-n^2\pi^2t/2}\cos(n\pi x_*)\cos(n\pi x),\qquad x\in[0,1],
\label{eq:c5-pde-density}
\end{equation}
so that $\mu_t(\dif x)=p(t,x)\dif x$; thus $\mu$ is the Neumann heat flow with $\mu_0=\delta_{x_*}$.

The data of the Markovian class are specified as follows.

\medskip
\noindent\textbf{(B1)} \quad The function
\[
f:[0,T]\times[0,1]\times\mathbb{R}\times\mathbb{R}\times\Pc_2(\mathbb{R}\times\mathbb{R})\to\mathbb{R}
\]
is Borel measurable and continuous in $(x,y,z,\nu)$ for almost every $t$. There exists a constant $L\ge0$ such that, for all $t,x$ and all $(y,z,\nu),(\widetilde y,\widetilde z,\widetilde\nu)$,
\begin{equation}
|f(t,x,y,z,\nu)-f(t,x,\widetilde y,\widetilde z,\widetilde\nu)|
\le L\bigl(|y-\widetilde y|+|z-\widetilde z|+W_2(\nu,\widetilde\nu)\bigr).
\label{eq:c5-pde-B1}
\end{equation}
Moreover,
\[
a(t):=\sup_{x\in[0,1]}|f(t,x,0,0,\delta_{(0,0)})|\in L^2(0,T).
\]

\medskip
\noindent\textbf{(B2)} \quad The boundary driver is $g(t,y)=-\gamma y$.

\medskip
\noindent\textbf{(B3)} \quad The terminal function $h\in C^2([0,1])$ satisfies the Robin compatibility conditions
\begin{equation}
h'(0)=\gamma h(0),\qquad h'(1)=-\gamma h(1).
\label{eq:c5-pde-compatible-h}
\end{equation}

\medskip
\noindent\textbf{(B4)} \quad The obstacle $u\in C([0,T])$ is deterministic and satisfies
\begin{equation}
\int_{[0,1]}h(x)\mu_T(\dif x)\ge u_T.
\label{eq:c5-pde-terminal-compatible}
\end{equation}

\begin{remark}
\label{rem:c5-pde-compatibility}
The compatibility conditions in \textbf{(B3)} allow us to apply the uniform gradient estimate \eqref{eq:c5-pde-compatible-gradient} for the terminal datum directly and to avoid a corner mismatch between the terminal datum and the boundary condition. We impose them for convenience.
\end{remark}

With the data above, equation \eqref{eq:mf-gbsde} reads
\begin{equation}
\begin{cases}
Y_t=h(X_T)+\int_t^Tf(s,X_s,Y_s,Z_s,\nu_s)\dif s
-\gamma\int_t^TY_s\dif\kappa_s+K_T-K_t-\int_t^TZ_s\dif B_s,\\[1ex]
\mathbb{E}[Y_t]\ge u_t,\qquad
\int_0^T(\mathbb{E}[Y_t]-u_t)\dif K_t=0,
\end{cases}
\label{eq:c5-pde-reference}
\end{equation}
where $\nu_s=\mathcal{L}(Y_s,Z_s)$. We begin with the first moment of the boundary local time, which controls the size of $\kappa$ in the a priori estimate.

\begin{lemma}[First moment of the boundary local time]
\label{lem:c5-pde-kappa}
For the process \eqref{eq:c5-pde-forward} and every $t\in[0,T]$,
\[
\mathbb{E}\kappa_t=t+\mathbb{E}[X_t(1-X_t)]-x_*(1-x_*)\le t+\frac14.
\]
In particular, the process $\kappa$ satisfies condition \textbf{(A5)} for the data \eqref{eq:c5-pde-reference}.
\end{lemma}

\begin{proof}
Apply It\^o's formula to $\phi(x)=x(1-x)$. Since $\phi'(0)=1$, $\phi'(1)=-1$ and $\phi''=-2$, the dynamics \eqref{eq:c5-pde-forward} gives
\[
\phi(X_t)=\phi(x_*)+\int_0^t(1-2X_s)\dif B_s-t+L^0_t+L^1_t.
\]
The stochastic integral is a martingale, because its integrand is bounded. Taking expectations yields the formula, and $\mathbb{E}[X_t(1-X_t)]\le1/4$.
\end{proof}

The next corollary records that the Markovian data fall under the scope of Theorem \ref{thm:existence_uniqueness}.

\begin{corollary}[Well-posedness of the reference equation]
\label{cor:c5-pde-reference}
Under \textbf{(B1)}--\textbf{(B4)}, equation \eqref{eq:c5-pde-reference} has a unique solution $(Y,Z,K)\in S^2\times H^2\times A_d^2$, and $Y\in L^2(q)$.
\end{corollary}

\begin{proof}
We check the conditions of assumption \textbf{(A)}. For \textbf{(A1)}, the composition
\[
\widetilde f(\omega,t,y,z,\nu):=f(t,X_t(\omega),y,z,\nu)
\]
is jointly measurable, and it is progressively measurable for every fixed $(y,z,\nu)$, because $X$ is continuous and adapted; the Lipschitz condition \eqref{eq:c5-pde-B1} and the bound on $f(t,\cdot,0,0,\delta_{(0,0)})$ transfer from \textbf{(B1)}. For \textbf{(A2)}, the boundary driver $-\gamma y$ satisfies the dissipativity condition with $\beta=-\gamma$ (that is, $b=\gamma$), together with the growth condition with $\psi=0$ and $L_g=\gamma$. Condition \textbf{(A3)} holds because $\xi=h(X_T)$ is bounded and $\mathbb{E}\xi=\int_{[0,1]}h\dif\mu_T\ge u_T$ by the terminal compatibility condition \eqref{eq:c5-pde-terminal-compatible}. Condition \textbf{(A4)} is the continuity requirement on $u$ in \textbf{(B4)}, and condition \textbf{(A5)} follows from Lemma \ref{lem:c5-pde-kappa}. Theorem \ref{thm:existence_uniqueness} now gives the result.
\end{proof}

\subsection{The nonlocal parabolic problem}
\label{sec:pde-problem}

Let $\mathcal{A}_\gamma$ be the Robin realization of $\frac12\partial_{xx}$ on $(0,1)$:
\[
\mathcal{A}_\gamma\phi=\frac12\phi'',\qquad
D(\mathcal{A}_\gamma)=\bigl\{\phi\in H^2(0,1):\phi'(0)=\gamma\phi(0),\ \phi'(1)=-\gamma\phi(1)\bigr\},
\]
where $H^2(0,1)$ is the Sobolev space, and let $S_\gamma(r)=e^{r\mathcal{A}_\gamma}$ be the associated semigroup; the same notation $S_\gamma$ is also used for its strongly continuous restriction to $C([0,1])$. We work in the space
\[
\mathcal{V}=\Bigl\{v\in C([0,T]\times[0,1]):v(t,\cdot)\in C^1([0,1])\text{ for a.e. }t,\ v_x\in L^2\bigl(0,T;C([0,1])\bigr)\Bigr\},
\]
equipped with the norm
\[
\|v\|_{\mathcal{V}}=\sup_{t,x}|v(t,x)|+\left(\int_0^T\|v_x(t,\cdot)\|_\infty^2\dif t\right)^{1/2}.
\]
Spatial derivatives are taken in the weak sense and, for almost every $t$, have continuous representatives. We fix a jointly measurable representative of $v_x$ with continuous spatial sections for almost every $t$, defined arbitrarily on the exceptional set of times; the flow $\nu^v$ is understood up to Lebesgue-null sets in time, and the process $v_x(t,X_t)$ is taken in its predictable $\dif t\otimes\mathbb{P}$-equivalence class. This is a convention and requires no additional assumption.

For $v\in\mathcal{V}$ and almost every $t$, define the law flow, the nonlocal driver and the spatial mean by
\[
\nu_t^v=\bigl(x\mapsto(v(t,x),v_x(t,x))\bigr)_\#\mu_t,
\qquad
H_v(t,x)=f(t,x,v(t,x),v_x(t,x),\nu_t^v),
\]
\[
m_v(t)=\int_{[0,1]}v(t,x)\mu_t(\dif x).
\]
By \textbf{(B1)}, $H_v\in L^2(0,T;C([0,1]))$: the uniform bound
\[
\|H_v(t)\|_\infty\le a(t)+L\bigl(\|v(t)\|_\infty+\|v_x(t)\|_\infty+W_2(\nu_t^v,\delta_{(0,0)})\bigr)
\]
holds, and $W_2(\nu_t^v,\delta_{(0,0)})^2\le\|v(t)\|_\infty^2+\|v_x(t)\|_\infty^2$. The function $m_v$ is continuous, because $v$ is uniformly continuous and $\mu$ is weakly continuous in $t$.

The PDE we associate with \eqref{eq:c5-pde-reference} is
\begin{equation}
\dif v(t,x)+\left[\frac12v_{xx}(t,x)+H_v(t,x)\right]\dif t+\dif K_t=0,
\qquad(t,x)\in(0,T)\times(0,1),
\label{eq:c5-pde-formal}
\end{equation}
together with the Robin boundary condition
\begin{equation}
v_x(t,0)=\gamma v(t,0),\qquad v_x(t,1)=-\gamma v(t,1),
\label{eq:c5-pde-robin}
\end{equation}
the terminal condition $v(T,x)=h(x)$, and the integral constraint
\begin{equation}
m_v(t)\ge u_t,\qquad\int_0^T(m_v(t)-u_t)\dif K_t=0.
\label{eq:c5-pde-constraint}
\end{equation}
The source term $\dif K_t$ is a measure in time; when $K$ is absolutely continuous with density $k$, equation \eqref{eq:c5-pde-formal} takes the form $\partial_tv+\frac12v_{xx}+H_v+k=0$. The constraint \eqref{eq:c5-pde-constraint} acts on the spatial mean $m_v$. By the Skorokhod condition, the multiplier measure $\dif K$ is supported on the contact set $\{t:m_v(t)=u_t\}$. No pointwise obstacle is imposed on $v$ itself.

Since the source in \eqref{eq:c5-pde-formal} is a general time measure, we define solutions through the variation-of-constants formula.

\begin{definition}[Mild mean-constrained solution]
\label{def:c5-pde-mild}
A pair $(v,K)$ is a mild mean-constrained solution of the system \eqref{eq:c5-pde-formal}--\eqref{eq:c5-pde-constraint} if:
\begin{enumerate}[label=(\roman*)]
\item $v\in\mathcal{V}$;
\item $K\in A_d^2$;
\item for every $t\in[0,T]$,
\begin{equation}
v(t,\cdot)=S_\gamma(T-t)h+\int_t^TS_\gamma(s-t)H_v(s,\cdot)\dif s
+\int_{(t,T]}S_\gamma(s-t)\mathbf{1}\dif K_s;
\label{eq:c5-pde-mild}
\end{equation}
\item the constraint \eqref{eq:c5-pde-constraint} holds.
\end{enumerate}
\end{definition}

The last term in the mild identity \eqref{eq:c5-pde-mild} differs from the spatial constant $K_T-K_t$: the Robin boundary interaction modifies the effect of the spatially constant source through $S_\gamma(s-t)\mathbf{1}$.

\subsection{Estimates for the Robin problem}
\label{sec:pde-estimates}

To solve the mild equation \eqref{eq:c5-pde-mild}, we need estimates for the Robin semigroup and for the linear problem with a time-measure source. The next lemma collects the contraction and gradient bounds used repeatedly below.

\begin{lemma}[Robin semigroup estimates]
\label{lem:c5-pde-semigroup}
The following statements hold:
\begin{enumerate}[label=(\roman*)]
\item $S_\gamma$ is a strongly continuous positive contraction semigroup on $C([0,1])$.
\item For $0<r\le T$,
\begin{equation}
\|\partial_xS_\gamma(r)\phi\|_\infty\le C_{\gamma,T}(1+r^{-3/4})\|\phi\|_\infty.
\label{eq:c5-pde-gradient}
\end{equation}
\item If $\phi\in C^2([0,1])$ satisfies the Robin conditions \eqref{eq:c5-pde-robin}, then
\begin{equation}
\sup_{r\ge0}\|\partial_xS_\gamma(r)\phi\|_\infty\le\max\{\|\phi'\|_\infty,\gamma\|\phi\|_\infty\}.
\label{eq:c5-pde-compatible-gradient}
\end{equation}
\item For every $r>0$,
\begin{equation}
\|\partial_xS_\gamma(r)\mathbf{1}\|_\infty\le\gamma.
\label{eq:c5-pde-constant-gradient}
\end{equation}
\end{enumerate}
\end{lemma}

\begin{proof}
The closed quadratic form
\[
\mathfrak a(\phi,\psi)=\frac12\int_0^1\phi'\psi'\dif x
+\frac\gamma2\bigl(\phi(0)\psi(0)+\phi(1)\psi(1)\bigr),
\qquad D(\mathfrak a)=H^1(0,1),
\]
is associated with the self-adjoint operator $-\mathcal{A}_\gamma$; integration by parts gives the domain of $\mathcal{A}_\gamma$ stated above. Since $H^1(0,1)$ embeds compactly into $L^2(0,1)$, the operator $-\mathcal{A}_\gamma$ has an orthonormal eigenbasis $(e_j)_{j\ge0}$ with $-\mathcal{A}_\gamma e_j=\lambda_je_j$ and $\lambda_j>0$, and the min--max principle, compared with the Neumann and Dirichlet forms, gives $\lambda_j\asymp(j+1)^2$, with constants depending on $\gamma$.

Each $e_j$ solves the constant-coefficient equation $\frac12e_j''=-\lambda_je_j$. Writing $q_j=\sqrt{2\lambda_j}$, we have
\[
e_j(x)\propto\cos(q_jx)+\frac{\gamma}{q_j}\sin(q_jx),
\]
and $q_j$ satisfies
\[
(\gamma^2-q_j^2)\sin q_j+2\gamma q_j\cos q_j=0.
\]
Normalization and the eigenvalue bounds give $\|e_j\|_\infty\le C_\gamma$ and $\|e_j'\|_\infty\le C_\gamma(j+1)$; for large $j$ this is immediate from the displayed form, and the finitely many remaining indices are absorbed into the constant.

For $\phi\in C([0,1])$, let $\phi_j=\int_0^1\phi e_j\dif x$. Then, by Cauchy--Schwarz,
\[
|\partial_xS_\gamma(r)\phi(x)|
\le\left(\sum_je^{-2\lambda_jr}|e_j'(x)|^2\right)^{1/2}\|\phi\|_{L^2}
\le C_{\gamma,T}(1+r^{-3/4})\|\phi\|_\infty,
\]
where the last step uses $\sum_j(j+1)^2e^{-c(j+1)^2r}\le C(1+r^{-3/2})$. This proves the gradient bound \eqref{eq:c5-pde-gradient}.

For smooth Robin-compatible $\phi$, It\^o's formula applied to $e^{-\gamma\kappa_s}S_\gamma(r-s)\phi(X_s)$ shows that
\[
S_\gamma(r)\phi(x)=\mathbb{E}_x\bigl[e^{-\gamma\kappa_r}\phi(X_r)\bigr];
\]
the boundary local-time terms cancel, because the spatial derivative of the semigroup satisfies the Robin condition at both endpoints. This representation gives positivity and $\|S_\gamma(r)\phi\|_\infty\le\|\phi\|_\infty$. Smooth Robin-compatible functions are dense in $C([0,1])$: one approximates a continuous function by smooth functions and then adjusts the endpoint derivatives in arbitrarily thin boundary layers. The contraction therefore extends to all continuous functions, and strong continuity follows first on the dense set and then on $C([0,1])$.

For the bound \eqref{eq:c5-pde-compatible-gradient}, put $w(r,x)=\partial_xS_\gamma(r)\phi(x)$. It solves the heat equation in the interior with initial datum $\phi'$ and boundary values $w(r,0)=\gamma S_\gamma(r)\phi(0)$ and $w(r,1)=-\gamma S_\gamma(r)\phi(1)$. The maximum principle and the contraction estimate prove the bound \eqref{eq:c5-pde-compatible-gradient}.

For the bound \eqref{eq:c5-pde-constant-gradient}, choose smooth Robin-compatible $\phi_n$ converging uniformly to $\mathbf{1}$, with $\|\phi_n'\|_\infty\le\gamma+o(1)$ and $\|\phi_n\|_\infty=1+o(1)$; such functions are obtained by adding boundary corrections of amplitude tending to zero, with endpoint slopes tending to $\gamma$ and $-\gamma$. Applying the bound \eqref{eq:c5-pde-compatible-gradient} to $\phi_n$ and passing to the limit in the eigenexpansion, where dominated convergence applies, gives the bound \eqref{eq:c5-pde-constant-gradient}.
\end{proof}

The next lemma solves the linear version of \eqref{eq:c5-pde-mild}, in which the nonlocal term $H_v$ is replaced by a prescribed function $H$, and provides the probabilistic identity along the forward process.

\begin{lemma}[Linear problem with a time-measure source]
\label{lem:c5-pde-source}
Let $H\in L^2(0,T;C([0,1]))$ and $K\in A_d^2$, and define
\begin{equation}
v(t)=S_\gamma(T-t)h+\int_t^TS_\gamma(s-t)H(s)\dif s+R_K(t),
\qquad
R_K(t)=\int_{(t,T]}S_\gamma(s-t)\mathbf{1}\dif K_s.
\label{eq:c5-pde-linear-mild}
\end{equation}
Then $v\in\mathcal{V}$, the Robin conditions \eqref{eq:c5-pde-robin} hold for almost every $t$, and
\begin{equation}
\|v\|_{\mathcal{V}}\le C_{\gamma,T}\bigl(\|h\|_{C^1}+\|H\|_{L^2(0,T;C([0,1]))}+K_T\bigr).
\label{eq:c5-pde-linear-bound}
\end{equation}
Moreover, along the forward process \eqref{eq:c5-pde-forward},
\begin{equation}
v(t,X_t)=h(X_T)+\int_t^TH(s,X_s)\dif s
-\gamma\int_t^Tv(s,X_s)\dif\kappa_s+K_T-K_t-\int_t^Tv_x(s,X_s)\dif B_s
\label{eq:c5-pde-linear-bsde}
\end{equation}
holds for all $t$, almost surely.
\end{lemma}

\begin{proof}
For the initial term, the compatibility bound \eqref{eq:c5-pde-compatible-gradient} gives
\[
\sup_{r\ge0}\|\partial_xS_\gamma(r)h\|_\infty\le\max\{\|h'\|_\infty,\gamma\|h\|_\infty\}.
\]
For the Lebesgue term, the contraction property bounds its supremum, and its spatial derivative is estimated with the gradient bound \eqref{eq:c5-pde-gradient} and Young's convolution inequality, since $r\mapsto C_{\gamma,T}(1+r^{-3/4})$ belongs to $L^1(0,T)$; hence this derivative lies in $L^2(0,T;C([0,1]))$. For the measure term, the constant-gradient bound \eqref{eq:c5-pde-constant-gradient} gives
\[
\|R_K(t)\|_\infty\le K_T-K_t,\qquad
\|\partial_xR_K(t)\|_\infty\le\gamma(K_T-K_t).
\]
The derivative under the Stieltjes integral is justified by the bound \eqref{eq:c5-pde-constant-gradient}, and no atom at $s=t$ is present because $K$ is continuous. Continuity of $R_K$ and of its spatial derivative follows by splitting the integral into the regions $s-t<\varepsilon$ and $s-t\ge\varepsilon$: the first piece is controlled by the modulus of continuity of $K$, while the second has a continuous kernel. The Lebesgue term is continuous by strong continuity of the semigroup, dominated convergence and approximation of $H$ in $L^2(0,T;C([0,1]))$. This proves $v\in\mathcal{V}$ and estimate \eqref{eq:c5-pde-linear-bound}. The Robin conditions follow term by term from the corresponding properties of the semigroup.

For the probabilistic identity \eqref{eq:c5-pde-linear-bsde}, choose smooth $H_n$ converging to $H$ in $L^2(0,T;C([0,1]))$. Extend $K$ constantly outside $[0,T]$, write the extension as $\overline K$, and set
\[
K^n_t=(\overline K*\rho_n)(t)-(\overline K*\rho_n)(0),\qquad n\ge1,
\]
where $(\rho_n)_{n\ge1}$ is a standard mollifier. Then $K^n$ is smooth and non-decreasing with $K^n_0=0$, $K^n\to K$ uniformly, and $0\le K^n_T\le K_T$; in particular, the finite measures $\dif K^n$ converge weakly to $\dif K$. Let $v_n$ be given by \eqref{eq:c5-pde-linear-mild} with data $(H_n,K^n)$. The smooth equation and It\^o's formula give the identity \eqref{eq:c5-pde-linear-bsde} for $v_n$; applying the formula up to $T-\varepsilon$ and then letting $\varepsilon\downarrow0$ avoids any additional terminal compatibility.

The semigroup estimates give $v_n\to v$ uniformly and $\partial_xv_n\to v_x$ in $L^2(0,T;C([0,1]))$. For the measure terms, uniform convergence of $K^n$ implies convergence of the Stieltjes integrals: away from the diagonal $s=t$, the kernels $S_\gamma(s-t)\mathbf{1}$ and $\partial_xS_\gamma(s-t)\mathbf{1}$ are uniformly continuous on compact sets, while near the diagonal they are bounded by $1$ and $\gamma$ by the bound \eqref{eq:c5-pde-constant-gradient}, and the measure of that region tends to zero uniformly in $t$. Consequently,
\[
H_n(s,X_s)\to H(s,X_s)\quad\text{in }L^2(\Omega\times[0,T]),
\qquad
\partial_xv_n(s,X_s)\to v_x(s,X_s)\quad\text{in }L^2(\Omega\times[0,T]),
\]
the stochastic integrals converge in the maximal $L^2$ norm, and the boundary integrals converge pathwise, since
\[
\sup_t\left|\int_0^t(v_n-v)(s,X_s)\dif\kappa_s\right|
\le\|v_n-v\|_\infty\kappa_T\to0.
\]
Taking a subsequence proves the identity simultaneously for all times.
\end{proof}

Replacing the nonlocal term by a prescribed flow $\nu$ leads to the following fixed-point problem, which will be used to construct the solution.

\begin{lemma}[Frozen-law semilinear problem]
\label{lem:c5-pde-frozen}
Let $\nu:[0,T]\to\Pc_2(\mathbb{R}\times\mathbb{R})$ be a measurable flow satisfying
\[
\varrho(t):=\left(\int_{\mathbb{R}\times\mathbb{R}}(|y|^2+|z|^2)\nu_t(\dif y,\dif z)\right)^{1/2}\in L^2(0,T).
\]
For every $K\in A_d^2$, the equation
\begin{equation}
v(t)=S_\gamma(T-t)h+\int_t^TS_\gamma(s-t)f(s,\cdot,v(s),v_x(s),\nu_s)\dif s+R_K(t)
\label{eq:c5-pde-frozen-mild}
\end{equation}
has a unique solution $v\in\mathcal{V}$. It satisfies
\begin{equation}
\|v\|_{\mathcal{V}}\le C_{L,\gamma,T}\bigl(\|h\|_{C^1}+\|a\|_{L^2(0,T)}+\|\varrho\|_{L^2(0,T)}+K_T\bigr).
\label{eq:c5-pde-frozen-bound}
\end{equation}
Moreover, the identity \eqref{eq:c5-pde-linear-bsde} holds with $H(s,\cdot)=f(s,\cdot,v(s),v_x(s),\nu_s)$.
\end{lemma}

\begin{proof}
The space $\mathcal{V}$ is complete. Indeed, if $(v_n)$ is Cauchy in $\mathcal{V}$, then $v_n$ converges uniformly and $\partial_xv_n$ converges in $L^2(0,T;C([0,1]))$; after passing to a subsequence, $v(t,x)-v(t,0)=\int_0^xw(t,r)\dif r$ for almost every $t$, where $w$ is the derivative limit, so the limit lies in $\mathcal{V}$.

For $\lambda\ge1$, define the equivalent norm
\[
\|v\|_\lambda=\sup_te^{-\lambda(T-t)}\|v(t)\|_\infty
+\left(\int_0^Te^{-2\lambda(T-t)}\|v_x(t)\|_\infty^2\dif t\right)^{1/2}.
\]
Let $\Phi$ be the right-hand side of \eqref{eq:c5-pde-frozen-mild}. Lemma \ref{lem:c5-pde-source} shows that $\Phi$ maps $\mathcal{V}$ into itself, because
\[
\|f(t,\cdot,v(t),v_x(t),\nu_t)\|_\infty
\le a(t)+L\bigl(\|v(t)\|_\infty+\|v_x(t)\|_\infty+\varrho(t)\bigr).
\]
For $v,w\in\mathcal{V}$, put $D(t)=f(t,\cdot,v(t),v_x(t),\nu_t)-f(t,\cdot,w(t),w_x(t),\nu_t)$. Then
\[
\|e^{-\lambda(T-\cdot)}D\|_{L^2(0,T;C([0,1]))}
\le L(\sqrt T+1)\|v-w\|_\lambda.
\]
The measure term cancels in the difference, so contraction of $S_\gamma$ and Cauchy--Schwarz give
\[
\sup_te^{-\lambda(T-t)}\|(\Phi v-\Phi w)(t)\|_\infty
\le(2\lambda)^{-1/2}\|e^{-\lambda(T-\cdot)}D\|_{L^2(0,T;C([0,1]))}.
\]
For the derivative, Young's convolution inequality and the gradient bound \eqref{eq:c5-pde-gradient} give
\[
\left(\int_0^Te^{-2\lambda(T-t)}\|\partial_x(\Phi v-\Phi w)(t)\|_\infty^2\dif t\right)^{1/2}
\le C_{\gamma,T}(\lambda^{-1}+\lambda^{-1/4})
\|e^{-\lambda(T-\cdot)}D\|_{L^2(0,T;C([0,1]))}.
\]
Hence
\[
\|\Phi v-\Phi w\|_\lambda
\le C_{L,\gamma,T}(\lambda^{-1/2}+\lambda^{-1}+\lambda^{-1/4})\|v-w\|_\lambda,
\]
and $\Phi$ is a strict contraction for sufficiently large $\lambda$. Banach's fixed point theorem gives existence and uniqueness, and the same estimate applied to $\Phi0$ gives the estimate \eqref{eq:c5-pde-frozen-bound}. The final assertion follows from Lemma \ref{lem:c5-pde-source} with $H(s,\cdot)=f(s,\cdot,v(s),v_x(s),\nu_s)$.
\end{proof}

\subsection{Well-posedness and probabilistic representation}
\label{sec:pde-wellposedness}

We can now state the well-posedness of the mean-constrained PDE and its probabilistic representation.

\begin{theorem}[Well-posedness and probabilistic representation]
\label{thm:c5-pde}
Under assumption \textbf{(B)}, there is a unique mild mean-constrained solution $(v,K)$ of the system \eqref{eq:c5-pde-formal}--\eqref{eq:c5-pde-constraint} in the sense of Definition \ref{def:c5-pde-mild}.

For the unique solution $(Y,Z,K)$ of the reference equation \eqref{eq:c5-pde-reference},
\begin{equation}
Y_t=v(t,X_t)\quad\text{for all }t,\qquad
Z_t=v_x(t,X_t)\quad(\dif t\otimes\mathbb{P})\text{-a.e.},
\label{eq:c5-pde-identification}
\end{equation}
and the law flow coincides with the one induced by $v$:
\[
\nu_t=\bigl(x\mapsto(v(t,x),v_x(t,x))\bigr)_\#\mu_t
\qquad\text{for a.e. }t.
\]
Let $M=\sup_{0\le t\le T}u_t^+$. There is a constant $C=C(L,\gamma,T)$ such that
\begin{equation}
\|v\|_{\mathcal{V}}^2+K_T^2
\le C\bigl(\|h\|_{C^1}^2+\|a\|_{L^2(0,T)}^2+M^2(1+T)\bigr).
\label{eq:c5-pde-bound}
\end{equation}
Finally, for every $(t,x)$, let $X^{t,x}$ be the reflected process started at $x$ at time $t$, with boundary integrator $\kappa^{t,x}$. Then
\[
Y_s^{t,x}=v(s,X_s^{t,x}),\qquad
Z_s^{t,x}=v_x(s,X_s^{t,x})
\]
solve, on $[t,T]$,
\begin{equation}
Y_s^{t,x}=h(X_T^{t,x})+\int_s^Tf(r,X_r^{t,x},Y_r^{t,x},Z_r^{t,x},\nu_r)\dif r
-\gamma\int_s^TY_r^{t,x}\dif\kappa_r^{t,x}+K_T-K_s-\int_s^TZ_r^{t,x}\dif B_r.
\label{eq:c5-pde-all-starts}
\end{equation}
The flow $\nu$ and the multiplier $K$ in \eqref{eq:c5-pde-all-starts} are those of the reference equation \eqref{eq:c5-pde-reference}; they are computed once and are kept fixed for all starting points.
\end{theorem}

\begin{proof}
By Corollary \ref{cor:c5-pde-reference}, the reference equation \eqref{eq:c5-pde-reference} has a unique solution $(Y,Z,K)$, and its law flow $\nu_s=\mathcal{L}(Y_s,Z_s)$ satisfies
\[
\int_0^T\varrho(s)^2\dif s
=\int_0^T\bigl(\mathbb{E}|Y_s|^2+\mathbb{E}|Z_s|^2\bigr)\dif s<\infty,
\]
because $Y\in S^2$ and $Z\in H^2$. Applying Lemma \ref{lem:c5-pde-frozen} to this flow and this $K$, we obtain $v\in\mathcal{V}$.

The evaluated pair $(\widehat Y_t,\widehat Z_t):=(v(t,X_t),v_x(t,X_t))$ solves the equation with the same fixed flow $\nu$ and multiplier $K$. This is the identity of Lemma \ref{lem:c5-pde-source} with the nonlocal driver in place of $H$. It belongs to $S^2\times H^2$, and $\widehat Y\in L^2(q)$ because $v$ is bounded and $\mathbb{E}\kappa_T<\infty$ by Lemma \ref{lem:c5-pde-kappa}.

We next use uniqueness for the equation with fixed flow and fixed multiplier. For two such solutions, It\^o's formula applied to $e^{\alpha t}|\delta Y_t|^2$ gives
\[
\begin{aligned}
\mathbb{E}\left[e^{\alpha t}|\delta Y_t|^2\right]
&+\mathbb{E}\int_t^Te^{\alpha s}|\delta Z_s|^2\dif s
+2\gamma\mathbb{E}\int_t^Te^{\alpha s}|\delta Y_s|^2\dif\kappa_s\\
&=\mathbb{E}\int_t^Te^{\alpha s}\bigl(2\delta Y_s\delta f_s-\alpha|\delta Y_s|^2\bigr)\dif s,
\end{aligned}
\]
where $\delta f_s=f(s,X_s,\widehat Y_s,\widehat Z_s,\nu_s)-f(s,X_s,Y_s,Z_s,\nu_s)$ and $\delta Y_T=0$. The stochastic integral is a uniformly integrable martingale by the BDG inequality, since $\delta Y\in S^2$ and $\delta Z\in H^2$. The Lipschitz condition \eqref{eq:c5-pde-B1} gives
\[
2\delta Y_s\delta f_s\le C_L|\delta Y_s|^2+\tfrac12|\delta Z_s|^2.
\]
Choosing $\alpha>C_L$ shows that $\delta Y=0$ and $\delta Z=0$. Hence $(\widehat Y,\widehat Z)=(Y,Z)$.

Consequently, the flow induced by $v$ coincides with $\nu$, and the mean constraint and the Skorokhod condition transfer from the reference equation to $v$. Thus $(v,K)$ is a mild mean-constrained solution.

For uniqueness, let $(v^1,K^1)$ and $(v^2,K^2)$ be mild mean-constrained solutions. By Lemma \ref{lem:c5-pde-source}, their evaluations along the reference process solve equation \eqref{eq:c5-pde-reference}, with their own law flows and multipliers. Uniqueness in Corollary \ref{cor:c5-pde-reference} gives $K^1=K^2$ and equal evaluations and law flows. Lemma \ref{lem:c5-pde-frozen} then gives $v^1=v^2$ on the whole cylinder.

For estimate \eqref{eq:c5-pde-bound}, the a priori estimate of Theorem \ref{thm:existence_uniqueness} and $\mathbb{E}\kappa_T\le T+\frac14$ give
\[
\|Y\|_{S^2}^2+\|Z\|_{H^2}^2+K_T^2
\le C_{L,\gamma,T}\bigl(\|h\|_\infty^2+\|a\|_{L^2(0,T)}^2+M^2(1+T)\bigr),
\]
which also bounds $\int_0^T\varrho(s)^2\dif s$. Substituting this bound into \eqref{eq:c5-pde-frozen-bound} gives the estimate \eqref{eq:c5-pde-bound}.

Finally, the proof of Lemma \ref{lem:c5-pde-source} applies to reflected Brownian motion started at any point $(t,x)$, and evaluating the frozen-law solution along $X^{t,x}$ gives the representation \eqref{eq:c5-pde-all-starts}.
\end{proof}

\begin{remark}[Interpretation of the representation]
\label{rem:c5-pde-representation}
Theorem \ref{thm:c5-pde} gives a nonlinear Feynman--Kac representation for a coupled system: the forward marginal flow $\mu$ is prescribed by the reflected diffusion, the backward field $v$ determines the joint law of its value and gradient under $\mu$, and the deterministic multiplier enforces the integral constraint. The Skorokhod condition restricts the support of the multiplier measure $\dif K$ to the contact set $\{t:m_v(t)=u_t\}$. The constraint acts on the spatial mean, while the field $v$ itself is not subject to a pointwise obstacle.
\end{remark}

\subsection{Weak formulation}
\label{sec:pde-weak}

The mild identity \eqref{eq:c5-pde-mild} is equivalent to a distributional formulation, in which the time-measure source acts through a single integral against the test function. The following proposition makes this precise.

\begin{proposition}[Weak formulation]
\label{prop:c5-pde-weak}
Let $v\in\mathcal{V}$, $K\in A_d^2$ and $H\in L^2(0,T;C([0,1]))$ satisfy the linear identity \eqref{eq:c5-pde-linear-mild}. Then, for every $t\in[0,T]$ and every $\phi\in C^2([0,1])$ satisfying the Robin conditions \eqref{eq:c5-pde-robin},
\begin{equation}
\begin{aligned}
\int_0^1v(t,x)\phi(x)\dif x
={}&\int_0^1h(x)\phi(x)\dif x\\
&+\int_t^T\left[\frac12\int_0^1v(s,x)\phi''(x)\dif x
+\int_0^1H(s,x)\phi(x)\dif x\right]\dif s\\
&+(K_T-K_t)\int_0^1\phi(x)\dif x.
\end{aligned}
\label{eq:c5-pde-weak}
\end{equation}
Conversely, the identity \eqref{eq:c5-pde-weak} for all such $\phi$ and all $t$ implies the linear identity \eqref{eq:c5-pde-linear-mild}. For the solution of Theorem \ref{thm:c5-pde}, one takes $H=H_v$; the Robin condition \eqref{eq:c5-pde-robin} then holds for almost every $t$.
\end{proposition}

\begin{proof}
Let $(e_j)_{j\ge0}$ and $(\lambda_j)_{j\ge0}$ be the eigenbasis and the eigenvalues of $-\mathcal{A}_\gamma$ from the proof of Lemma \ref{lem:c5-pde-semigroup}. Testing the mild identity \eqref{eq:c5-pde-linear-mild} against $e_j$ and using self-adjointness gives, for every $t$,
\[
v_j(t)=e^{-\lambda_j(T-t)}h_j+\int_t^Te^{-\lambda_j(s-t)}H_j(s)\dif s
+c_j\int_{(t,T]}e^{-\lambda_j(s-t)}\dif K_s,
\]
where $v_j(t)=\int_0^1v(t,x)e_j(x)\dif x$, $h_j=\int_0^1h(x)e_j(x)\dif x$, $H_j(s)=\int_0^1H(s,x)e_j(x)\dif x$ and $c_j=\int_0^1e_j(x)\dif x$. Hence, as an identity of measures in time,
\[
\dif v_j(t)=\lambda_jv_j(t)\dif t-H_j(t)\dif t-c_j\dif K_t.
\]
Integrating from $t$ to $T$ gives the identity \eqref{eq:c5-pde-weak} for $\phi=e_j$. A general Robin-compatible $\phi\in C^2([0,1])$ is a limit of finite spectral sums in the graph norm of $\mathcal{A}_\gamma$; since $v$ and $H$ are square-integrable in space and time and $K_T<\infty$, passing to the limit yields the identity \eqref{eq:c5-pde-weak}.

Conversely, if the identity \eqref{eq:c5-pde-weak} holds for every Robin-compatible $\phi$ and every $t$, then testing against each $e_j$ gives a scalar terminal-value problem with a finite-measure forcing, whose unique variation-of-constants solution is the formula above. Completeness of the eigenbasis yields the identity in $L^2(0,1)$ for every $t$, and both sides have continuous representatives in $x$, so the identity \eqref{eq:c5-pde-linear-mild} holds pointwise in $x$.
\end{proof}

\subsection{The multiplier: regularity and singularity}
\label{sec:pde-multiplier}

We now study the regularity of the multiplier. The reflected forward process gives an explicit expression for the expected boundary integral, which we record first.

\begin{lemma}[Expected boundary local time]
\label{lem:c5-pde-boundary-density}
For every continuous deterministic $w:[0,T]\times[0,1]\to\mathbb{R}$,
\begin{equation}
\mathbb{E}\int_0^Tw(t,X_t)\dif\kappa_t
=\frac12\int_0^T\bigl[w(t,0)p(t,0)+w(t,1)p(t,1)\bigr]\dif t.
\label{eq:c5-pde-boundary-density}
\end{equation}
Moreover, the boundary densities satisfy
\[
p(t,0)+p(t,1)\le C_T(1+t^{-1/2}),\qquad0<t\le T.
\]
\end{lemma}

\begin{proof}
The density estimate follows from the series \eqref{eq:c5-pde-density} and $\sum_{n\ge1}e^{-cn^2t}\le C(1+t^{-1/2})$. For $\phi\in C^2([0,1])$, It\^o's formula gives
\[
\mathbb{E}[\phi(X_t)]-\phi(x_*)
=\frac12\int_0^t\mathbb{E}[\phi''(X_s)]\dif s
+\phi'(0)\mathbb{E}[L^0_t]-\phi'(1)\mathbb{E}[L^1_t].
\]
On the other hand, the Neumann heat equation for $p$ and integration by parts give
\[
\mathbb{E}[\phi(X_t)]-\phi(x_*)
=\frac12\int_0^t\int_0^1\phi''(x)p(s,x)\dif x\dif s
+\frac12\int_0^t\bigl[\phi'(0)p(s,0)-\phi'(1)p(s,1)\bigr]\dif s.
\]
Comparing the two expressions and using the independence of $\phi'(0)$ and $\phi'(1)$, together with the integrability of $p(s,0)$ and $p(s,1)$ near $s=0$, we obtain
\[
\mathbb{E}[L^i_t]=\frac12\int_0^tp(s,i)\dif s,\qquad i=0,1.
\]
These identities determine the expected Stieltjes measures $\mathbb{E}[\dif L^0_t]$ and $\mathbb{E}[\dif L^1_t]$; the support properties of $L^0$ and $L^1$ then give the identity \eqref{eq:c5-pde-boundary-density}, first for step functions in time and then for continuous functions.
\end{proof}

For the solution of Theorem \ref{thm:c5-pde}, define
\begin{equation}
D(t)=\int_0^1H_v(t,x)p(t,x)\dif x
-\frac\gamma2\bigl[v(t,0)p(t,0)+v(t,1)p(t,1)\bigr],
\qquad0<t\le T.
\label{eq:c5-pde-mean-drift}
\end{equation}
The interior term is bounded by $\|H_v(t)\|_\infty$, and the boundary term is integrable by Lemma \ref{lem:c5-pde-boundary-density}; hence $D\in L^1(0,T)$.

Combining the boundary density with the reference equation yields a scalar backward Skorokhod problem for the spatial mean. The following proposition solves it and describes the multiplier.

\begin{proposition}[Backward Skorokhod formula]
\label{prop:c5-pde-skorokhod}
Let $(v,K)$ be the solution of Theorem \ref{thm:c5-pde}, and write $m(t)=m_v(t)$ and $m_T=\int_{[0,1]}h\dif\mu_T$. Then
\begin{equation}
m(t)=m_T+\int_t^TD(s)\dif s+K_T-K_t.
\label{eq:c5-pde-mean-equation}
\end{equation}
Define
\[
Q(t)=\sup_{r\in[t,T]}\left[u_r-m_T-\int_r^TD(s)\dif s\right]^+.
\]
Then
\begin{equation}
K_t=Q(0)-Q(t),\qquad K_T=Q(0).
\label{eq:c5-pde-K-sup}
\end{equation}
If $u\in W^{1,1}(0,T)$, then $K\in W^{1,1}(0,T)$. Writing $k=K'$, one has, for almost every $t$,
\begin{equation}
k(t)=\mathbf{1}_{\{m(t)=u_t\}}\bigl(-u'(t)-D(t)\bigr)^+,
\qquad
0\le k(t)\le|u'(t)|+|D(t)|.
\label{eq:c5-pde-K-density}
\end{equation}
Consequently,
\[
\partial_tv+\frac12v_{xx}+H_v+k=0
\]
holds in the distributional sense with respect to Lebesgue measure in time.
\end{proposition}

\begin{proof}
Taking expectations in the reference equation \eqref{eq:c5-pde-reference} and applying Lemma \ref{lem:c5-pde-boundary-density} to $w=v$ yields the mean equation \eqref{eq:c5-pde-mean-equation}. All terms are integrable: the interior driver is integrable in time, the boundary term has integrable expectation by Lemma \ref{lem:c5-pde-boundary-density}, and the stochastic integral has zero expectation.

Put $x(t)=m_T+\int_t^TD(s)\dif s-u_t$. Then $x$ is continuous, $x(T)\ge0$, and
\[
m(t)-u_t=x(t)+K_T-K_t.
\]
The deterministic backward Skorokhod problem for $x$ has the solution
\[
K_T-K_t=\sup_{r\in[t,T]}[-x(r)]^+=Q(t).
\]
Indeed, the right-hand side is continuous and non-increasing in $t$, vanishes at $t=T$, and dominates $-x(t)$, so the constraint holds; its associated increasing measure is supported on the set where $x(t)+K_T-K_t=0$, because where this sum is strictly positive, continuity separates the present value from the future running maximum, which is locally constant.

Uniqueness of this representation can be seen without the explicit formula. For two solutions with the same $x$, set $r^i(t)=x(t)+K_T^i-K^i(t)$. Then $\delta r(T)=0$ and $\dif\delta r=-\dif\delta K$, so
\[
|\delta r(t)|^2=2\int_t^T\delta r(s)\dif\delta K_s.
\]
The two Skorokhod conditions and $r^i\ge0$ make the right-hand side non-positive, hence both solutions coincide. This proves the representation \eqref{eq:c5-pde-K-sup}.

If $u\in W^{1,1}(0,T)$, then $x$ is absolutely continuous with $x'=-D-u'$ almost everywhere. The future running maximum of an absolutely continuous function is absolutely continuous: if $q(t)=\sup_{r\in[t,T]}[-x(r)]^+$, then, for $a<b$,
\[
0\le q(a)-q(b)\le\int_a^b|x'(s)|\dif s.
\]
To see the bound, note that if the maximum over $[a,T]$ is attained after $b$, then the difference is zero; otherwise, compare the maximizing value on $[a,b]$ with the value at $b$. Therefore $K$ is absolutely continuous and $k\le|D+u'|$ almost everywhere.

Now $r=m-u$ is absolutely continuous and non-negative. An absolutely continuous function has derivative zero almost everywhere on a level set; in particular, $r'=0$ almost everywhere on $\{r=0\}$. From the mean equation \eqref{eq:c5-pde-mean-equation},
\[
r'=-D-k-u'.
\]
The Skorokhod condition gives $k=0$ almost everywhere on $\{r>0\}$; on $\{r=0\}$ the derivative identity gives $k=-D-u'\ge0$. This proves the density formula \eqref{eq:c5-pde-K-density}. Replacing $\dif K_t$ by $k(t)\dif t$ in the weak formulation gives the final display.
\end{proof}

The time-measure formulation is unavoidable in general: the next result exhibits data for which the multiplier is singular continuous.

\begin{proposition}[A continuous obstacle with a singular multiplier]
\label{prop:c5-pde-singular}
There exist data satisfying \textbf{(B)} for which the multiplier $K$ of the unique mild mean-constrained solution is singular continuous. For those data, the PDE \eqref{eq:c5-pde-formal} admits no classical solution that is continuously differentiable in time.
\end{proposition}

\begin{proof}
Take $f=0$ and $h=0$. Let $K$ be the rescaled Cantor distribution function on $[0,T]$, so that $K_0=0$, $K_T=1$, and $\dif K$ is continuous and singular with respect to Lebesgue measure. Define
\[
v(t)=R_K(t)=\int_{(t,T]}S_\gamma(s-t)\mathbf{1}\dif K_s,
\qquad
u_t=\int_{[0,1]}v(t,x)\mu_t(\dif x).
\]
By Lemma \ref{lem:c5-pde-source}, $v\in\mathcal{V}$ and the probabilistic identity \eqref{eq:c5-pde-linear-bsde} holds. The obstacle $u$ is continuous, $u_T=0$, and the mean constraint holds with equality at every time. Hence $(v,K)$ is a mild mean-constrained solution of the problem with these data, and Theorem \ref{thm:c5-pde} shows that it is the unique one.

Let $e_0$ be the positive first eigenfunction of $-\mathcal{A}_\gamma$ and set $c_0=\int_0^1e_0(x)\dif x>0$. The scalar coefficient $v_0(t)=\int_0^1v(t,x)e_0(x)\dif x$ satisfies
\[
\dif v_0(t)=\lambda_0v_0(t)\dif t-c_0\dif K_t.
\]
Since $c_0>0$ and $\dif K$ has a non-zero singular part, the distributional derivative of $v_0$ has a non-zero singular part, so $v_0$ is not absolutely continuous. In particular, $v$ is not continuously differentiable in time, and no classical solution with an ordinary Lebesgue-time source represents these data.
\end{proof}

\end{document}